\documentclass[11pt]{amsproc}

\usepackage[latin1]{inputenc}
\usepackage[english]{babel}
\usepackage{mathrsfs}
\usepackage{amsmath}
\usepackage{amsfonts}
\usepackage{amsthm}
\usepackage{latexsym}
\usepackage[dvips]{graphics,color}
\usepackage{graphicx} 
\usepackage{amssymb}
\usepackage{units}
\usepackage{mathrsfs}
\numberwithin{equation}{section}
\newcommand{\ind}[1]{\textbf I_{#1}}
\newcommand{\ww}{\textbf I_{n+1}^{\text{\tiny{WW}}}}
\newcommand{\wb}{\textbf I_{n+1}^{\text{\tiny{WB}}}}
\newcommand{\bb}{\textbf I_{n+1}^{\text{\tiny{BB}}}}
\newcommand{\hww}{\textbf I^{\text{\tiny{WW}}}}
\newcommand{\hwb}{\textbf I^{\text{\tiny{WB}}}}
\newcommand{\hbb}{\textbf I^{\text{\tiny{BB}}}}
\newcommand{\wh}{\textbf I_{n+1}^{\text{\tiny{W}}}}
\newcommand{\bl}{\textbf I_{n+1}^{\text{\tiny{B}}}}
\newcommand{\geqla}[1]{\stackrel{*#1}{\geq}}
\newcommand{\leqla}[1]{\stackrel{*#1}{\leq}}
\newcommand{\fa}{\mathscr F}
\newcommand{\ex}{\mathbb E}
\newcommand{\Var}{\mbox{Var}}
\newcommand{\pr}{\mathbb P}
\newcommand{\real}{\mathbb R}
\newcommand{\bigo}{\mathcal O}
\newcommand{\sgn}{\text{sgn}}
\newcommand{\tmin}{t_{\mbox{\tiny min}}}
\newcommand{\tmax}{t_{\mbox{\tiny max}}}

\newcommand{\e}{\mathcal E}
\newcommand{\N}{\mathcal N}
\newcommand{\D}{\mathcal D}
\newcommand{\floor}[1]{\left\lfloor{#1}\right\rfloor}

\newcommand{\acc}{X_\infty}
\newtheorem{theorem}{Theorem}
\newtheorem{lemma}{Lemma}

\newtheorem{rem}{Remark}
\newtheorem{deff}{Definition}

\title[Generalized Pólya urns via stochastic approximation]
{Generalized Pólya urns via stochastic approximation}

\author{Henrik Renlund}
\address{Department of Mathematics, Uppsala University, PO Box 480,
S-751 06 Uppsala, Sweden}
\email{henrik.renlund@math.uu.se}
\urladdr{http://www.math.uu.se/$\sim$renlund/}


\date{}
\begin{document}

\maketitle

\begin{abstract}
$\phantom .$ \\ 
We collect, survey and develop methods of (one-dimensional) 
stoch\-astic approximation in a framework
that seems suitable to handle fairly broad generalizations of Pólya urns.

To show the applicability of the results
we determine the limiting fraction of balls in an urn with balls of two colors.
We consider two models generalizing the Pólya urn, in the first one ball is
drawn and replaced with balls of (possibly) both colors according to which color
was drawn. In the second, two balls are drawn simultaneously and replaced along
with balls of (possibly) both colors according to what combination of colors
were drawn. 

\end{abstract}

\tableofcontents
\newpage

\pagenumbering{arabic}

\setcounter{page}{1}

\section{Introduction}

\subsection{Urns}
The urn is a common tool in probability theory and statistics and no student thereof 
can avoid it. Imagine an urn with $w$ white and $b$ black balls. 
At a beginners level, urns provide examples of how to calculate 
probabilities, e.g.\ the probability of drawing a white ball is the number of
white balls divided by the total number of balls, i.e.\ $w/(w+b)$. If we sample
more than one ball, say $n$ balls, from the urn and count the number of white ball we 
get examples of the binomial distribution (with parameters
$n$ and $w/(w+b)$) and hypergeometric distributions (with parameters
$w+b, n$ and $w/(w+b)$, depending on whether we sample with or without replacement. 
These distributions in turn are very important in statistical theory as they are
the key to understanding properties of surveys, e.g.\ voter polls,
such as margins of error. 

More aspects of probability theory can be illustrated via urns. Suppose we draw
two balls without replacement. The question ``what is the probability that the second ball is white?''
may introduce the concept of conditional probabilities, as the
answer depends on the knowledge we have (or lack) regarding the
outcome of the first draw. Urns are so useful that it is hard to imagine
an introductory text on probability and statistics not ever mentioning urns of any kind.
Any reader with a general interest in urns may consult \cite{JK77}.

In 1923 Eggenberger and Pólya introduced a new urn model in \cite{Polya},
now commonly referred to as a Pólya urn. An
urn has one white and one black ball. We sample one ball and replace it along
with one additional ball of the same color, and repeat this procedure. 
It was thought of as a simple model for a contagious disease.
The first draw might correspond to a doctor examining the very first patient of the day.
She then has a $50\%$ risk of being infected. Now, the essence of a contagious disease
is that the more people have it, the more likely you are to get it, and
vice versa. This is now 
reflected in the model in the following way. Say white ball means ``infected''.
After we draw a white ball we replace at along with one additional white
ball. Hence, the probability of drawing a white ball next time has risen to $2/3\approx 67\%$. It
basically means that the more infected patients the doctor gets, the more likely it
is that there are yet more to come. Of course, the
actual numbers in this example is by no means meant to be ``realistic'', 
it is rather a qualitative model. 

We can, however, play with the parameters of the model to better fit some
specific situation if needed. First, the initial composition of the urn
need not be 1 of each color. A \emph{rare}  disease might correspond
to $10\,000$ black balls and only $1$ white. Also, some diseases are more
contagious than others. We could incorporate this by stating that we should 
not add \emph{one} additional ball, but several, of the same color
as the one drawn, corresponding to a faster
spread of the disease. 

Any reader interested in Pólya urns and generalizations thereof
can start with \cite{Mah08}.

Our own interest in Pólya-like urn models comes from a similar 
situation as described above but rather than modelling infectious diseases, it
can model how something is learned, e.g.\ a ``brain'' trying 
to learn what to do in a specific situation. Assume for simplicity that 
there are only two possible ways to act, act 1 and act 2, and
that act 1 is the correct way to handle the situation and,
as such, leads to a reward of some kind. Act 2 is wrong and has no benefit
for our brain. However, at first it is not 
known to our brain which act is correct (if any). It must somehow learn this by trial and error.
 A very simple urn model describing how this brain could work
is the following. To model an initial state of ignorance, there is 
one white ball (meaning ``do act 1'') and one black ball (meaning
``do act 2'') so that the first time it just picks one ball (act) randomly.
Then, to model reinforcement learning, there is a  rule that if an act is deemed successful,
 more balls of the color corresponding to the act just performed
are added to the urn. In this case; if a white ball is drawn,
add, say, one additional white ball and if a black ball is drawn replace
it but add no more balls. Now, every time our brain performs 
the right act it becomes increasingly
likely that it will do so again. 

As with the previous model, the interest is mainly qualitative. One should
not expect that any brain works exactly like an urn. However, it captures
some of the dynamics of what one can think of as learning; one tends to
be more likely to do things that have proved successful in the past. 

Again, we can fine tune the parameters. More colors can mean more
ways to act, different reinforcement rules between colors can
specify how much benefit the brain gets from the different acts,
and so on. 

More specifically, it was questions 
relating to the so called
``signaling problems'' (communicated by Persi Diaconis and Brian Skyrms) 
that spawned the authors interest
in these matters. These refer to the situation where two (or more) agents
try to acquire a common language simultaneously via urns.
Recently, one of these problems was solved in \cite{APSV08}
which also contains a more thorough description of the problem.

This is some of the motivation behind studying urns evolving 
along the lines of  ``draw one or several balls and add
more balls according to
some prescribed rule depending on the colors of the drawn balls''. 
It is also the motivation for only looking at the fraction
of balls, as these dictate the probabilities of ``acting correctly''
in models of learning.

\subsection{Stochastic approximation algorithms}
A stochastic approximation algorithm $\{X_n\}$ is usually defined
as an $\real^d$-valued stochastic process adapted to a filtration $\{\fa_n\}$ such that
\begin{equation}\label{saa} X_{n+1}=X_n+\gamma_{n+1}[f(X_n)+\epsilon_{n+1}] \end{equation}
holds, where the decreasing ``steplengths'' $\gamma_n>0$ satisfy
$\sum_n \gamma_n=\infty$ and $\sum_n \gamma_n^2<\infty$.
The random variables $\gamma_n$ 
can be considered stochastic or deterministic but in either
case it is usually assumed that $\{\gamma_n\epsilon_n\}$ is
a martingale difference sequence, i.e.\ 
\begin{equation} \label{krav}
\ex[\gamma_n\epsilon_n|\fa_{n-1}]=0.
\end{equation}
The origin of this subject is \cite{RM51}, in which Robbins and
Monro considered the following
one-dimensional
problem; suppose that given an input $x$ to some
system in which we get $M(x)$ as output, where $M$ is an unknown
function and 
only observable through white noise. What we really observe
is thus $M(x)+\epsilon$, for some random variable $\epsilon$ with $\ex\epsilon=0$. 
We want to find the input $\theta$ so that
$M(\theta)=\alpha$ for some prescribed $\alpha$. 
For simplicity we might assume that $M$ is nondecreasing and that
$M(x)=\alpha$ has a unique solution $\theta$. 

A candidate algorithm for finding a sequence 
$\{X_n\}$ that converges (in some sense) to $\theta$ is
to start with some initial input $X_0=x_0$. Given a
value $X_n$, with $n\geq 0$, create the next element by
\[ X_{n+1}=X_n+\frac1{n+1}(\alpha-M(X_n)+\epsilon_{n+1}), \]
where $-\epsilon_n$ is the noise associated with the $n$'th
observation. The algorithm works on an intuitive level
since whenever $X_n\neq \theta$ then, on average, $X_{n+1}$
takes a step in the direction of $\theta$. 

This describes a stochastic approximation algorithm with
drift function $f(x)=\alpha-M(x)$ and steplengths
$\gamma_{n}=1/n$. Of course, there is nothing
in the formulation of the problem that demands us 
to set the steplengths to $1/n$. To demand
$\sum_n \gamma_n=\infty$ is natural since this
basically  means that the algorithm can wander
arbitrarily far, thus hopefully finding what it is looking for,
and not converging in a trivial manner.

Next, since
\[ X_n-x_0=\sum_{k=1}^n\gamma_k(f(X_{k-1})+\epsilon_{k}), \]
the requirement $\sum_n \gamma_n^2<\infty$ makes 
$\Var X_n$ bounded (under additional assumptions
on the error terms and $f$).

In the multidimensional case
the heuristics behind the algorithm (\ref{saa}) is that
it constitutes
a discrete time version of the ordinary differential equation
\begin{equation} \label{ODE}
 \frac{d}{dt}x_t=f(x_t),
\end{equation}
subject to ``noise''. If the noise vanishes for large $n$ it seems
plausible that the interpolation of $X_n$  should estimate
some trajectory of a solution $x_t$ of (\ref{ODE}), an idea made
precise in \cite{Ben99}, where more references may be found.
An overview may  also be found in \cite{Pem07}.
We are however only concerned with the one-dimensional case.

Any reader interested in other aspects of stochastic approximation
and applications may find \cite{Bor08} useful.

\subsection{How they fit}
Stochastic approximation is very well suited for urn models with
reinforcements such as the classical Pólya urn and generalizations
thereof. If a ball is drawn
from an urn and (a bounded number of) balls are added according to some reinforcement
scheme, the difference of the proportion of balls before and after
is approximately some function of the proportion times  $1/n$.

As an example, consider the so called Friedman's urn starting with one ball
each of two colors where $a>0$ balls of the same color and $b>0$ balls of the other
color are added along with the ball drawn. The proportion $Z_n$
of  either color then satisfies
\[ Z_{n+1}-Z_n=\frac{1}{2+(n+1)(a+b)}\big[f(Z_n)+\text{``noise''}\big], \]
with the drift function $f(Z_n)=b(1-2Z_n)$ and where ``noise'' is a 
martingale difference sequence. This resembles the situation
considered by Robbins and Monro and, as the drift always points towards
$1/2$, it seems intuitive that this is the point of convergence of 
$Z_n$ (in some sense). That this is so will follow from 
Theorem \ref{main} below. This  is ``easy'' since 1/2 is the
unique solution of $f(x)=0$. 

In other urn models $f(x)=0$ may have several roots. There are known
results that deal with multiple zeros, although often under
the property (\ref{krav}). Urn schemes where the total number of 
balls added each time is not constant tend to lose this property. 
We will generalize existing results under an assumption slightly weaker
than (\ref{krav}) and apply the results to
generalized Pólya urns.

\subsection[A generalized Pólya urn]{A generalized Pólya urn considered as a \\ stochastic approximation algorithm}
\label{1-draget}
First, we will show more precisely how stochastic approximation algorithms fit urn schemes
by presenting an application which will be studied in more detail below.
Consider an urn with balls of two colors, white and black say. Let $W_n$ and $B_n$ denote
the number of balls of each color, white and black respectively, 
 after the $n$'th draw and consider the initial values $W_0=w_0>0$ and
$B_0=b_0>0$ to be fixed. After each draw we notice the color and replace it along with
additional balls according to the \emph{replacement matrix}
\begin{align*}  & \quad\; \text{W}  \,\;\; \text{B} \nonumber \\
\begin{array}{c}
    \text{W} \\ \text{B} 
   \end{array} &
\left( \begin{array}{c c}
        a & b \\ c & d 
       \end{array} \right), 
\begin{array}{r}
 \mbox{where}\;\min\{a,b,c,d\}\geq 0\phantom{,} \\
\mbox{and}\;\max\{a,b,c,d\}>0,
\end{array} \end{align*}
so that, e.g.\ a white ball is replaced along with $a$ additional white and $b$ additional black
balls. We demand that $a,b,c,d$ are nonnegative numbers. 

This model is by no means new, chapter 3 of \cite{Mah08} gives a historical overview.
Setting $a=d=1$, $b=c=0$ and $W_0=B_0=1$ gives the classical Pólya urn described in 
the introduction.

We let $\wh$ and $\bl$ denote the indicators of getting a white and black ball in draw
$n$, respectively. We set $T_n=W_n+B_n$ and $Z_n=W_n/T_n$. Recursively, $W_n$ and $T_n$ evolve as
\begin{equation*}
 W_{n+1} = W_n+a\wh+c\bl \quad\quad\mbox{and}\quad\quad T_{n+1}=T_n+(a+b)\wh+(c+d)\bl
\end{equation*}
and hence, with $\Delta Z_n=Z_{n+1}-Z_n$, 
\begin{align*}
 \Delta Z_n&=\frac{1}{T_{n+1}}\big[W_n+a\wh+c\bl-Z_n(T_n+(a+b)\wh+(c+d)\bl)\big] \\
&=\frac{1}{T_{n+1}}[\wh(a-(a+b)Z_n)+\bl(c-(c+d)Z_n)]=\frac{Y_{n+1}}{T_{n+1}}.
\end{align*}
Let  $\fa_n$ denote the history of the process up to time $n$, i.e.\ the
$\sigma$-algebra $\sigma(X_1,\ldots,X_n)$.
We will define 
\begin{align*}
 f(Z_n)=\ex[Y_{n+1}|\fa_n]&=Z_n(a-(a+b)Z_n)+(1-Z_n)(c-(c+d)Z_n) \\
&=\alpha Z_n^2+\beta Z_n+c,
\end{align*}
where
\begin{equation*}
 \alpha=c+d-a-b\quad\quad\mbox{and}\quad\quad\beta=a-2c-d.
\end{equation*}
In the form of a stochastic approximation algorithm we can write
\[ \Delta Z_n=\gamma_{n+1}\big[f(Z_n)+U_{n+1}\big], \]
where $U_{n+1}=Y_{n+1}-f(Z_n)$ and $\gamma_{n+1}=1/T_{n+1}$. 

Now, $U_{n+1}$ is mean-zero ``noise'' but if
$a+b\neq c+d$ then in general $\ex_n \gamma_{n+1}U_{n+1}\neq 0$. However,
as will be shown later, $|\ex[\gamma_{n+1}U_{n+1}|\fa_n]|=\bigo(T_n^{-2})$,
and $T_n$ (usually) grows like $n$, so this conditional expectation
is vanishing fast.

\section{The method of stochastic approximation}
We will apply the stochastic approximation machinery to fractions and thus limit
ourselves to processes in $[0,1]$. This naturally restricts the noise 
and the function to be bounded. 

\subsection{Definition}
Stochastic variables are throughout  assumed to be defined on a
probability space $(\Omega,\fa,\pr)$, although we will find no need
to make any reference to the underlying measurable space
$(\Omega,\fa)$. We will also consider a filtration $\{\fa_n, n\geq 0\}$
to be given. 

To simplify notation,  let $\ex_n(\cdot)=\ex(\cdot|\fa_n)$ and
$\pr_n(\cdot)=\pr(\cdot|\fa_n)$
denote
the conditional expectation and probability, respectively, 
with respect to $\fa_n$. 

\begin{deff}\label{def} $\phantom{1}$\\
A stochastic approximation algorithm $\{X_n\}$ is a 
stochastic process taking values in $[0,1]$, adapted
to the filtration $\{\fa_n\}$, that satisfies
\begin{equation}\label{defeq}
  X_{n+1}-X_n=\gamma_{n+1}[f(X_n)+U_{n+1}],
\end{equation}
where $\gamma_n,U_n\in\fa_n, f:[0,1]\to\real$ and 
the following conditions hold a.s.
\begin{itemize}
 \item[(i)] $c_l/n\leq \gamma_n\leq c_u/n$,
 \item[(ii)] $|U_n|\leq K_u$,
 \item[(iii)] $|f(X_n)|\leq K_f$, and
 \item[(iv)] $|\ex_n(\gamma_{n+1}U_{n+1})|\leq K_e\gamma_n^2$,
\end{itemize}
where the constants $c_l,c_u,K_u,K_f,K_e$ are positive real numbers. 
For future reference, set $K_\Delta=c_u(K_f+K_u)$.
\end{deff}

\begin{rem}
There is no consensus in the scientific litterature as to exactly what 
constitutes a stochastic approximation algorithm. The main characteristic
is that a relation of type (\ref{defeq}) holds, although the range, measurability
etc.\ of the ingredients $\gamma_n, U_n$ and $f$ may differ. In this section we
state results concerning ''the'' process $\{X_n\}$ which throughout is
understood to be a stochastic approximation algorithm according
to our definition.
\end{rem}

\begin{rem}
 The condition (iv) could, in view of condition (i), equally well have been
formulated as $|\ex_n(\gamma_{n+1}U_{n+1})|\leq K_e'n^{-2}$, for some
positive constant $K_e'$. The formulation above arises naturally for
the applications toward the end of this paper.

Condition (iv) replaces the more common requirement (\ref{krav}), so that $\gamma_nU_n$ 
does not necessarily have conditional expectation 0, but this
expectation is tending to zero quickly.
In what follows, we verify that some results known to be true for condition
(\ref{krav}) carry over to the present situation, as well as present some new results.
\end{rem}

\subsection{Limit points}
In this section we establish that the accumulation points of the
process $\{X_n\}$ are a subset of the zeros of $f$, for continuous
$f$. This property is well known  and the 
ideas for the proofs of Lemma \ref{convergence}
and Lemma \ref{tozero}  are from \cite{Pem07}.  
Moreover, Theorem \ref{main} gives an existence result for the
limit of the process $\{ X_n\}$.

Let $W_n=\sum_1^n \gamma_kU_k$ so that we may write increments of the process $\{X_n\}$ as
\begin{equation*}
 X_{n+k}-X_n=\sum_{k=n+1}^{n+k}\gamma_kf(X_{k-1})+W_{n+k}-W_n. 
\end{equation*}

\begin{lemma}\label{wn}
 $\{W_n\}$ converges almost surely.
\end{lemma}

\begin{proof}
 Set $Y_k=\gamma_kU_k$ and $\tilde Y_k=\ex_{k-1}(\gamma_kU_k)$ and define the martingale 
$M_n=\sum_{1}^{n}(Y_k-\tilde Y_k)$. Then
\[ \ex M_n^2=\ex\left\{\sum_{1}^{n}(Y_k-\tilde Y_k)^2\right\}\leq 
 \sum_{1}^{n}\ex Y_k^2\leq \sum_{1}^{n}\frac{c_u^2K_u^2}{k^2}<\infty \]
so that $M_n$ is an $L^2$-martingale and thus convergent. Next, since
\[  \sum_{1}^\infty\left|\tilde Y_k\right|\leq \sum_{1}^\infty \frac{c_u^2K_e}{(k-1)^2}<\infty \]
we must  also have that $\sum_{k\geq 1}Y_k$ converges a.s.
\end{proof}

\begin{deff} Let
\begin{equation*}
 \acc=\bigcap_{n\geq 1} \overline{\{X_n,X_{n+1},\ldots\}}
\end{equation*}
be the set of accumulation points
of $\{X_n\}$. 
\end{deff}

\begin{lemma} \label{convergence}
Suppose that $f(x)<-\delta$ (or $f(x)>\delta$), for some $\delta>0$, whenever $x\in(a_0,b_0)$.
Then \[ \acc\cap(a_0,b_0)=\emptyset\quad\mbox{a.s.}  \] and
either $\limsup_n X_n\leq a_0$ or $\liminf_n X_n\geq b_0$.
\end{lemma}
\begin{proof}
The proof follows that of Lemma 2.6 of \cite{Pem07}.

Let $[a,b]\subset(a_0,b_0)$ and let $\Delta=\min\{a-a_0,b_0-b\}$ be the smallest distance
from $[a,b]$ to a point outside $(a_0,b_0)$. Let $N>4c_uK_f/\Delta$ be a (random) number large enough 
so that
$n,m\geq N$ implies
\[ |W_n-W_m|<\Delta/4, \]
which by Lemma \ref{wn} is possible a.s.\ due to the a.s.\ convergence of $W_n$. Then we have
for any $n\geq N$
\begin{equation*}
 X_{n+1}-X_n=\gamma_{n+1}f(X_n)+W_{n+1}-W_n<\Delta/2,
\end{equation*}
so that the process after $N$ cannot immediately go from a point to 
the left of $a_0$ to a point on the right of $a$. 
Also, if $n\geq N$, $X_n\in(a_0,b]$ and $X_{n+1},\ldots,X_{n+k-1}\in(a_0,b_0)$
then
\begin{align*}
 X_{n+k}-X_n&=\sum_{j=n+1}^{n+k}\gamma_{j}f(X_{j-1})+W_{n+j}-W_n \\
&<-\delta\sum_{j=n+1}^{n+k}\gamma_{j}+\Delta/4.
\end{align*}
The last step shows that after $N$ the process cannot increase by more than
$\Delta/4$ while inside $(a_0,b_0)$, hence cannot escape out
to the right.  Moreover, since $\sum_{k>N}\gamma_k\to\infty$ a.s.,
we must have $X_{N+k^*}<a_0$ for some $k^*>0$. 

Now, once the process is below $a_0$ it will never reach above $a_0+\Delta/2$
in one step. Once inside $(a_0,b_0)$ it will never increase by more than
$\Delta/4$. Hence, it will never again reach above $a_0+3\Delta/4<a$. 
Obviously, we a.s.\ cannot have both $\liminf_n X_n\leq a$
\emph{and} $\limsup_n X_n\geq b$.

The first results follows from choosing $[a_k,b_k]\subset(a_0,b_0)$, such that \\
$\cup_k [a_k,b_k]=(a_0,b_0)$, so that 
\[ \pr\{ \acc\cap(a_0,b_0)\neq\emptyset\}\leq\sum_k \pr\{ X_n\in[a_k,b_k]\mbox{\ i.o.}\}=0. \]
The second results follows by an analogous calculation, yielding
\[ \pr\left(\{\liminf_{n\to\infty} X_n\leq a_0\}\cap \{\limsup_{n\to\infty} X_n\geq  b_0\}\right)=0, \] 
and the observation that
since we thus must have $\liminf_n X_n>a_0$ or $\limsup_n X_n<b_0$, we must in fact have
either a.s.\ 
\[ \liminf_{n\to\infty} X_n\geq b_0 \quad\mbox{or}\quad\limsup_{n\to\infty} X_n \leq a_0, \] 
since no accumulation points
exist in $(a_0,b_0)$ by the first result.

The case where $f(x)>\delta$ on $(a_0,b_0)$ is analogous.
\end{proof}
Next, we introduce the concept of attainability that we need now and again 
to rule out trivialities.
\begin{deff}\label{attain}
Call a subset $I$ \emph{attainable} if for every fixed $N\geq 0$ there exists an $n\geq N$ such that
\[ \pr(X_n\in I)>0.\] 
\end{deff}
Any ''reasonable'' stochastic approximation algorithm on $[0,1]$
should have $f(0)\geq0$ and $f(1)\leq 0$, otherwise it seems that the drift could push the processes 
out of $[0,1]$. The notion of attainability gives a sufficient condition to ensure
this. 
\begin{lemma}\label{fbra}
Assume that the drift function $f$ is continuous at the boundary points
$0$ and $1$.
If all neighborhoods of the origin are attainable, then $f(0)\geq 0$.
Similarly, if all neighborhoods of\, $1$ are attainable, then $f(1)\leq 0$.
\end{lemma}
We postpone the proof of this as it will be included in the proof of Theorem \ref{stable}
on page \pageref{case2}.

\begin{lemma}\label{tozero}
Suppose $f$ is continuous and let $Q_f=\{x:f(x)=0\}$ the zeros  of $f$. 
Then \[ \pr\big\{\acc\subseteq Q_f\big\}=1. \] 
\end{lemma}

\begin{proof}
The continuity of $f$ 
makes the sets \[ A_n=\{x\in(0,1): f(x)>1/n\;\mbox{or}\;f(x)<-1/n\}\] open. Hence,
each $A_n$ is a countable union of open intervals, each on which $f$ is $>1/n$ or $<-1/n$
and hence where no accumulation points may exist. 

The only ''loose end'' here is the boundary. Suppose e.g.\ that $f<0$ close to zero (but
a priori not at zero). Then it seems
that the process might be pushed down to zero (or beyond) even though $0\notin Q_f$. This
is however ruled out by Lemma \ref{fbra}, since if neighborhoods of the origin are
attainable then $f(0)\geq 0$ and if they are not, then the process eventually is bounded
away from the origin. Similarly we can not have $f>0$ close to  $x=1$ and attainability
of this boundary point simultaneously, unless $f(1)=0$.

It is clear that if $f>0$ close to the origin then the process will eventually be bounded
away from there (and similarly if $f<0$ close to $x=1$ then the process will be bounded
away from $1$).
\end{proof}

\begin{theorem}\label{main}
 If $f$ is continuous then $\displaystyle \lim_{n\to\infty} X_n$ exists a.s.\ and is in $Q_f$.
\end{theorem}
\begin{proof}
If $\lim_n X_n$ does not exist, we can find two different rational numbers
in the open interval $(\liminf_n X_n, \limsup_n X_n)$.

Let $p<q$ be two arbitrary different rational numbers. If we can show that
\begin{equation*} 
\pr\left(\{\liminf_{n\to\infty} X_n\leq p\}\cap \{\limsup_{n\to\infty} X_n\geq q\} \right)=0, 
\end{equation*}
the existence of the limit will be established and the 
claim of the theorem will follow from Lemma \ref{tozero}.

To do this we need to distinguish between whether or not $p$ and $q$ are in the
same connected component of $Q_f$.

\vspace{0.3cm}
\noindent\textbf{Case 1:} $p$ and $q$ are in not in the same connected component of $Q_f$. \\
Since $Q_f$ is closed and $f$ continuous, there must exist $(a,b)\subseteq (p,q)\cap Q_f^c$ such that 
$f$ is non-zero and of the same sign for all $x\in(a,b)$. By Lemma \ref{convergence}
it is impossible to have $\liminf_n X_n\leq a$ and $\limsup_n X_n\geq b$.

\vspace{0.3cm}
\noindent\textbf{Case 2:} $p$ and $q$ are in the same connected component of $Q_f$. \\
Assume that $\liminf_n X_n\leq p$ and fix an arbitrary $\epsilon$ in such
a way that $0<\epsilon<q-p$. We aim to show
that $\limsup_n X_n\leq p+\epsilon$. 

Recall the notation $W_n=\sum_1^n \gamma_kU_k$. We know from
Lemma \ref{wn} that $W_n$ converges a.s., so for some stochastic $N>2K_\Delta/\epsilon$,
we have that $n,m\geq N$ implies $|W_n-W_m|<\epsilon/2$. By assumption there
is some stochastic $n\geq N$ such that $X_n-p<\epsilon/2$.

Set 
\begin{align*} 
\tau_1=\inf\{k\geq n:X_k\geq p\}\quad\mbox{and}\quad
\sigma_{1}=\inf\{k>\tau_1: X_k<p\}
\end{align*}
and define, for $n\geq 1$,
\begin{equation*} 
\tau_{n+1}=\inf\{k>\sigma_n:X_k\geq p\}\quad\mbox{and}\quad\sigma_{n+1}=\inf\{k>\tau_n: X_k< p\}.
\end{equation*}
Now, for all $k$,
\begin{align}\label{apan1} 
X_{\tau_k}&=X_{\tau_k-1}+\Delta X_{\tau_k-1}\leq p+K_\Delta/\tau_k<p+\epsilon/2.
\end{align}
Note that $f(x)=0$ when $x\in[p,q]$. Hence, if $\tau_k+j-1$
is a time before the exit time of the interval $[p,q]$ then
\[ X_{\tau_k+j}=X_{\tau_k}+\sum_{\tau_k+1}^{\tau_k+j}\gamma_if(X_{i-1})+W_{\tau_k+j}-W_{\tau_k}
 =X_{\tau_k}+W_{\tau_k+j}-W_{\tau_k}. \]
As 
\begin{equation}\label{apan2}|W_{\tau_k+j}-W_{\tau_k}|<\epsilon/2 \end{equation} 
the process will never be able to reach
above $p+\epsilon$ before $\sigma_{k+1}$. Since (\ref{apan1})
and (\ref{apan2}) is true for all $k$, we must have
$\sup_{k\geq n} X_k \leq p+\epsilon$.  
\end{proof}

\subsection{Categorizing equilibrium points}
Any point $x\in Q_f=\{x:f(x)=0\}$ is called an equilibrium point, or zero, of $f$.
In this paper we shall use the following terminology:
\begin{itemize}
\item  A point $p\in Q_f$ is called \emph{unstable} if there exists a neighborhood
$\mathscr N_p$ of $p$ such that $f(x)(x-p)\geq 0$ whenever $x\in\mathscr N_p$.

This means that $f(x)\geq 0$ when $x$ is just above $p$ and $f(x)\leq 0$ when
$x$ is just below $p$, hence the drift is locally pushing the process away from 
$p$ (or not pushing at all). 

If $f(x)(x-p)>0$ when $x\in\mathscr N_p\backslash\{p\}$ we  call $p$
\emph{strictly unstable}. If $f$ is differentiable then $f'(p)>0$ is sufficient
to determine that $p\in Q_f$ is strictly unstable.

\item A point will be called \emph{stable} if there exists a neighborhood
$\mathscr N_p$ of $p$ such that $f(x)(x-p)<0$ whenever 
$x\in\mathscr N_p\backslash\{p\}$. If $f$ is differentiable then $f'(p)<0$
is sufficient to determine that $p\in Q_f$ is stable.

Locally, the drift pushes the
process towards $p$ from both directions. 

\item A point $p\in Q_f\cap(0,1)$ is called a \emph{touchpoint} if there
exists a neighborhood $\mathscr N_p$ of $p$ such that either $f(x)>0$ 
for all $x\in \mathscr N_p\backslash\{p\}$ or 
$f(x)<0$ for all $x\in \mathscr N_p\backslash\{p\}$. 
If $f$ is twice differentiable then $f(p)=f'(p)=0$ and $f''(p)\neq 0$ is sufficient
to determine that $p\in(0,1)$ is a touchpoint.

A touchpoint may be thought of as having one stable and one strictly unstable
side. Note that our definition does not allow touchpoints on the boundary.
\end{itemize}

\subsection{Nonconvergence}
In this section we narrow down the set of limit points of the process by excluding certain 
unstable points. 

\subsubsection[Unstable points]{Unstable points with non-vanishing error terms}
Here we exclude the unstable zeros of $f$ as possible limit points,
given that the error terms do not vanish at these points.
For our applications below this is applicable to zeros of $f$ in $(0,1)$ as the noise
does vanish at the boundary $\{0,1\}$, a problem addressed in the next section.

Heuristically, the process $\{X_n\}$ may arrive at an unstable point $p\in(0,1)$
by ``accident''. To ensure that it does not stay there, we need to know that there
is enough noise to push the process out into the drift leading away from $p$.

The main result here, Theorem \ref{Pem} below, is
an adaptation of Theorem 3.5 of \cite{Pem88}, a sketch of which can
be found in \cite{Pem07} and a corresponding multidimensional result
in \cite{Pem90},  whereby condition
(\ref{krav}) is replaced by (iv). 
For results on nonconvergence to more general unstable sets in the 
multidimensional case the reader is referred to section 9 of
\cite{Ben99} and references there.

To begin with we mention a result which will be used.
\begin{lemma} \label{con}
Let $A\in\fa_\infty=\sigma(\cup_n\fa_n)$ and suppose there is some integer $N$ 
and a real number $0<a\leq 1$ such that
$n>N$ implies $\pr(A|\fa_n)\geq a$. Then $\pr(A)=1$.
\end{lemma}
\begin{proof}
 The sequence $\pr(A|\fa_n)=\ex_n(\ind A)$ is an a.s.\ convergent martingale and $\lim_n \ex_n(A)=\ex(\ind A|\fa_\infty)=
\ind A$ a.s., see Th.\ 35.6 of \cite{Bil95}. If this variable is bounded away from zero it must be 1.
\end{proof}
Also, the following will prove to be useful.
\begin{lemma}\label{zwetc}
Let $N\geq 0$ be an integer and $\tau$ be a stopping time with respect
to the filtration in Definition \ref{def}, such
that $\tau\geq N$ a.s. Let $A\in\fa_\infty, B=A^c$,
\begin{align*}
Z_k= Z_k(N,\tau)&=[\ex_{k-1}\Delta X_{k-1}-\Delta X_{k-1}]\ind{\{N<k\leq \tau\}},\quad\mbox{and}\\
W_{m} &= \sum_{k=N+1}^m Z_k.
\end{align*}
Suppose that on $A$ we have $W_{\tau}>0$ or that we on $A$ have
$W_\tau<0$,
then 
\begin{equation}\label{hilfen}
\ex_N^2[W_{\tau}|A]\frac{\pr_N(A)}{\pr_N(B)}\leq \ex_NW_{\tau}^2
\leq \frac{K_\Delta^2}{N}.
\end{equation}
\end{lemma}

\begin{proof}
 First, we note that for any $m>N$
\[ \ex_N W_{m}^2\leq  \sum_{k=N+1}^\infty\ex_N\big[(\Delta X_{k-1})^2\big]\leq \sum_{k=N+1}^\infty 
\frac{K_\Delta^2}{k^2}\leq \frac{K_\Delta^2}{N}<\infty, \]
so that $W_{m}$ is an $L^2$-martingale and hence a.s.\ convergent. Due to the  assumption
that on $A$ we have $W_\tau$ strictly positive, or strictly negative, we must have $\pr_N(A)<1$, 
otherwise we would have $0\neq \ex_N W_\infty=\ex_N W_\tau$.
In particular, this assumption means that $\pr_N(B)>0$ so that we can make
the following calculation
\begin{align*}
 0 = \ex_N W_{\infty}=\ex_N[W_{\tau}]&=\ex_N[W_{\tau}|A]\pr_N(A)+\ex_N[W_{\tau}|B]\pr_N(B)  \\ \Longleftrightarrow\quad
-\ex_N[W_{\tau}|B] &=\ex_N[W_{\tau}|A]\frac{\pr_N(A)}{\pr_N(B)} \\ \Longrightarrow\quad
\ex_N^2[W_{\tau}|B] &=\ex_N^2[W_{\tau}|A]\left[\frac{\pr_N(A)}{\pr_N(B)}\right]^2.
\end{align*}
 Next, since $\ex X^2\geq \ex^2 X$ is true for any random variable $X$,
\begin{align*}
 \ex_N[W_{\tau}^2] &=\ex_N[W_{\tau}^2|A]\pr_N(A)+\ex_N[W_{\tau}^2|B]\pr_N(B) \\
&\geq \ex_N^2[W_{\tau}|A]\pr_N(A)+\ex_N^2[W_{\tau}|A]\frac{\pr_N^2(A)}{\pr_N(B)} \\
&=\ex_N^2[W_{\tau}|A]\pr_N(A)\left(1+\frac{\pr_N(A)}{\pr_N(B)}\right) \\
&=\ex_N^2[W_{\tau}|A]\frac{\pr_N(A)}{\pr_N(B)}. 
\end{align*}
\end{proof}

\begin{theorem} \label{Pem}
Assume that there exist an unstable point $p$ in $Q_f$, 
i.e.\ such that $f(x)(x-p)\geq 0$ locally, and that
\begin{equation}\label{v}
 \ex_n U_{n+1}^2\geq K_L
\end{equation}
holds, for some $K_L>0$, whenever $X_n$ is close to $p$. Then 
\[ \pr\{X_n\to p\}=0. \]
\end{theorem}

\begin{rem}\label{local}
The local assumptions $f(x)(x-p)\geq 0$ and (\ref{v}) can without loss of generality be assumed, in the
proof, to hold globally. Assume that the theorem is proved with global assumptions
but that $f(x)(x-p)\geq 0$ and (\ref{v}) are only satisfied when $X_n$ is in a neighborhood
$\mathscr N_p$ of $p$. 
Couple the process $\{X_n\}$ after a late time $N$
to another process $\{Y_n\}$, such that $Y_N=X_N$ and
\[ \Delta Y_n=\Delta X_n\ind{\{n>N, X_n\in\mathscr N_p\}}+\Delta Y_n'\ind{\{n>N, X_n\notin\mathscr N_p\}}. \]
If $\{Y_n',n>N\}$ is constructed so as to satisfy
 the global assumptions of Theorem \ref{Pem}, then so
does $Y_n$. Now if $\pr(X_n\to p)>0$, then 
 the same would be true of $\{Y_n\}$, contradicting the theorem.
\end{rem}

\begin{proof}[Proof of Theorem \ref{Pem}]
 Following Pemantle's proof there are two steps that need verification:\\
\noindent\textbf{Step 1:} Show that there is a $\beta>0$ such that for all $N$ large enough 
\begin{equation}\label{step1}
 \pr_N\bigg[\sup_{k\geq N}|X_k-p|>\beta/\sqrt N\bigg]\geq 1/2.
\end{equation} 
\noindent\textbf{Step 2:} 
Let
\begin{equation}\label{stopp} \tau=\inf\{k\geq N: |X_k-p|>\beta/\sqrt N\}. \end{equation}
Conditional on $\{\tau<\infty\}$ show that
\begin{equation}\label{step2}
 \pr_\tau\bigg[\inf_{k\geq \tau}|X_k-p|\geq \beta/2\sqrt N\bigg]\geq a, 
\end{equation}
for some $a>0$ not depending on $N$.

If  (\ref{step1}) and (\ref{step2}) are true then
\[ \pr_N(p\notin\acc)\geq\pr_N(\tau<\infty)
\pr_{\tau}\bigg(\sup_{k\geq\tau}|X_k-p|>\beta/2\sqrt N\;\big|\{\tau<\infty\}\bigg)\geq \frac a2>0, \]
and the result follows from Lemma \ref{con}.

\textbf{Notation:} Throughout the proof we will justify inequalities
(as they appear in calculations)
by stating that they hold if a parameter is sufficiently large. We
will denotes this by $\leqla n, \geqla n$ or $\leqla{}, \geqla{}$ 
if the inequality holds
if $n$ is sufficiently large, or if it 
is clear from the context which parameter is referred to, respectively. 
E.g.\ $\frac{10}n+\frac1{\sqrt n} \leqla{}  \frac2{\sqrt n}$, 
since this is true if $n\geq 100$.  

\vspace{0.3cm}
\noindent\textbf{Verification of Step 1:} \\
First, in view of Remark \ref{local}, we assume that $f(x)(x-p)\geq 0$ and
$\ex_n U_{n+1}^2\geq K_L$ holds globally.

We aim to show that  $\pr_N\{\tau=\infty\}\leq 1/2$ where $\tau$ is defined in (\ref{stopp}).

Recall that $K_\Delta=c_u(K_f+K_u)$, so that we have the bounds
\[ |\Delta X_n|\leq \frac{K_\Delta}{n+1}\quad\mbox{and}\quad(\Delta X_n)^2\leq \frac{K_\Delta^2}{(n+1)^2}. \]

We may assume that $\tau>N$, otherwise there is nothing to prove.
Examine the process $|X_{\tau\wedge m}-p|^2$ for $m>N$. An upper bound on
this quantity is given by
\begin{align*} 
 |X_{\tau\wedge m}-p| &= |X_{\tau\wedge m-1}-p+\Delta X_{\tau\wedge m-1}|\\
&\leq \frac{\beta}{\sqrt N}+\frac{K_\Delta}{\tau\wedge m}\leq \frac{\beta}{\sqrt N}+\frac{K_\Delta}{N} 
\leqla{ } 2\frac{\beta}{\sqrt N},
\end{align*}
and so  
\begin{align}\label{upper}
 G_n(m)=\ex_N[(X_{\tau\wedge m}-p)^2]
\leq \frac{4\beta^2}{N}.
\end{align}
Next, we make use of the relation 
\begin{align*}
(X_{\tau\wedge m}&-p)^2 = [X_{\tau\wedge (m-1)}-p+\Delta X_{m-1}\ind{\tau\geq m}]^2 \\
&= (X_{\tau\wedge (m-1)}-p)^2+2(X_{\tau\wedge (m-1)}-p)\Delta X_{m-1}\ind{\tau\geq m} \\
&\quad\quad\quad\quad\quad\quad\quad\quad\;\,  +(\Delta X_{m-1})^2\ind{\tau\geq m}.
\end{align*}
Since $m>N$ we have $\fa_N\subset \fa_{m-1}$ so any conditional 
expectation $\ex_N(\cdot)$ can be calculated
as $\ex_N\ex_{m-1}(\cdot)$. Hence, 
\begin{align}\label{gn}
G_N(m) =G_N(m-1) &+2\ex_N\big\{\ind{\tau\geq m}(X_{m-1}-p)\ex_{m-1}\big[\Delta X_{m-1}\big]\big\} \nonumber \\
&+ \phantom 2\ex_N\big\{\ind{\tau\geq m}\ex_{m-1}\big[(\Delta X_{m-1})^2\big]\big\}. 
\end{align}
Now, by the assumption $\ex_nU_{n+1}^2\geq K_L$ we get
\begin{align}\label{low}
 \ex_{m-1}[(\Delta  & X_{m-1})^2] =\ex_{m-1}[\gamma_{m}^2(f(X_{m-1})+U_m)^2] \nonumber \\
&\geq \frac{c_l}{m}\ex_{m-1}[\gamma_{m}f^2(X_{m-1})+\gamma_{m}U_{m}^2+2f(X_{m-1})\gamma_{m}U_{m}] \nonumber \\
&\geq \frac{c_l}{m}\bigg[\frac{c_l}{m}f^2(X_{m-1})+\frac{c_l}{m}\ex_{m-1}U_{m}^2
-2|f(X_{m-1})|\cdot|\ex_{m-1}\gamma_{m}U_{m}|\bigg] \nonumber \\
&\geq \frac{c_l^2K_L}{m^2}-\frac{c_lc_u^22K_fK_e}{m(m-1)^2} \nonumber \\
&\geqla{m} \frac{c_l^2K_L}{2m^2}. 
\end{align}
Also, by the assumption $f(x)(x-p)\geq 0$ we have that
\begin{align}\label{lower}
(X_{m-1}&-p)\ex_{{m-1}}[\Delta X_{m-1}] \nonumber\\
&	=(X_{m-1}-p)f(X_{m-1})\ex_{m-1}\gamma_m+(X_{m-1}-p)\ex_{m-1}\gamma_mU_m \nonumber \\
 &\geq 0-\frac{|X_{m-1}-p|K_ec_u^2}{(m-1)^2}.
\end{align}
We can now get a lower bound on $G_N(m)$. Continuing (\ref{gn}), using (\ref{low}) and 
combining (\ref{lower}) with the fact that $|X_{m-1}-p|<\beta/\sqrt N$ when $N<m\leq\tau$,
we see that
\begin{align}\label{grec}
 G_N(m)&\geqla{N} G_N(m-1)+\frac{c_l^2K_L}{2m^2}\ex_N\{\ind{\tau\geq m}\}
      -2\frac{c_u^2K_e\beta}{\sqrt N(m-1)^2}\ex_N\big\{\ind{\tau\geq m}\big\} \nonumber \\
&\geqla{N} G_N(m-1)+\frac{c_l^2K_L}{4m^2}\pr_N\{\tau\geq m\} \nonumber \\
&\geq G_N(m-1)+\frac{c_l^2K_L}{4m^2}\pr_N\{\tau= \infty\},
\end{align}
where the last inequality is true for any $m$ since $\{\tau\geq m\}\supset \{\tau=\infty\}$.
Expanding this recursion gives us 
\begin{align*}
G_N(m) &\geq G_N(N)+\frac14c_l^2K_L\pr_N\{\tau=\infty\}\sum_{k=N+1}^m\frac{1}{k^2} \\
&\geq G_N(N)+\frac14c_l^2K_L\pr_N\{\tau=\infty\}\left(\frac1{N+1}-\frac{1}{m+1}\right).
\end{align*}
Letting $m\to\infty$ and combining this with (\ref{upper}) we have
\begin{align*}
\pr_N(\tau=\infty)&\leq \frac{16\beta^2}{c_l^2K_L}\frac{N+1}{N}
\leq\frac{32\beta^2}{c_l^2K_L}.
\end{align*}
Choosing $\beta\leq\sqrt{c_l^2K_l/64}$ makes $\pr_N(\tau=\infty)\leq 1/2$.

\vspace{0.3cm}
\noindent\textbf{Verification of Step 2:}\\
Assume throughout that  
$\{\tau<\infty\}$, $\tau$ defined by (\ref{stopp}), is
realized through the event $\{X_\tau>p+\beta/\sqrt N\}$. The case
when $\{X_\tau<p-\beta/\sqrt N\}$ is similar. Set
\[ \hat\tau=\inf\{k\geq \tau: X_k<p+\beta/2\sqrt N\}. \]
We aim to show that $\pr_\tau\{\hat\tau=\infty\}\geq a$, with $a>0$.

With notation as in Lemma \ref{zwetc} let
$A=\{\hat\tau<\infty\}$ and set
$Z_k=Z_k(\tau,\hat\tau)$. Notice that by conditioning on $\tau$
we may consider it fixed (so that Lemma 5 is indeed applicable). 

Observe that by the assumption $f(x)(x-p)\geq 0$ we must have
$f(X_{k-1})\geq 0$ when $\tau<k\leq\hat\tau$, since $X_{k-1}-p>0$ in this
case. This gives us 
\begin{equation*}
 \ex_{k-1}\Delta X_{k-1}=f(X_{k-1})\ex_{k-1}\gamma_{k}+\ex_{k-1}\gamma_{k}U_{k}\geq -\frac{c_u^2K_e}{(k-1)^2}
\end{equation*}
and hence on the event $A=\{\hat\tau<\infty\}$, 
\begin{align*}
 W_{\hat\tau}=\sum_{\tau+1}^{\hat\tau} Z_k&=\sum_{\tau+1}^{\hat\tau}\ex_{k-1}\Delta X_{k-1}
    -\sum_{\tau+1}^{\hat\tau}\Delta X_{k-1} \\
&\geq  -\sum_{\tau+1}^{\hat\tau}\frac{c_u^2K_e}{(k-1)^2}-(X_{\hat\tau}-X_\tau) \\
&\geq -\frac{c_u^2K_e}{\tau-1}-p-\frac{\beta}{2\sqrt N}+p+\frac{\beta}{\sqrt N}\\
&\geq \frac{\beta}{2\sqrt N}-\frac{c_u^2K_e}{N-1}\geqla{N} \frac{\beta}{4\sqrt N}.
\end{align*}
Lemma \ref{zwetc} now gives us
\[\frac{\pr_N(\hat\tau=\infty)}{\pr_N(\hat\tau<\infty)}\geq \frac{\ex_N^2 [W_{\hat\tau}|\hat\tau<\infty]}{K_\Delta^2/\tau}\geq
 \frac{\tau\beta^2}{N16K_\Delta^2}\geq\frac{\beta^2}{16K_\Delta^2} =a'>0,\]
which implies $\pr_\tau(\hat\tau=\infty)\geq a'/(1+a')=a>0$. 
\end{proof}

\subsubsection{Strictly unstable boundary points}
In this section we deal with strictly unstable zeros on the boundary. These
present a new problem as the error
terms tend to vanish, making Theorem \ref{Pem} inapplicable. 
This new result motivated a separate paper \cite{Ren09}. 

Interestingly, the key ingredient here is an upper bound on how
fast the error terms are vanishing when the process gets near
the unstable point on the boundary. This is quite the opposite 
to the situation in Theorem \ref{Pem}, which required a lower bound on the
error terms. This may at first  seem odd. However, the heuristics
is that if the process cannot arrive at the boundary in a finite
number of steps, knowing that the error terms get small enough means
an increasing tendency for the process to follow the drift. 

\begin{theorem}\label{renlund}
Suppose of the process $\{X_n\}$ from Definition \ref{def} 
that $X_n\in(0,1)$ for all $n$.
Assume that $p\in\{0,1\}\cap Q_f$ is such that $f(x)(x-p)> 0$
whenever $x\neq p$ is close to $p$
and that there are positive constants $K_f', K_u'$
such that a.s.
\begin{align}
 \ex_n U_{n+1}^2 &\leq K_u'|X_n-p|, \label{ukrav} \\
 [f(x)]^2 &\leq K_f'|x-p|, \quad\mbox{and}\label{fkrav}\\
 k\cdot |X_k-p| &\to\infty,\quad\mbox{as}\;\;k\to\infty. \label{xkrav}
\end{align}
Then $\pr\{X_n\to p\}=0$. 
\end{theorem}

\begin{rem}\label{notstop}
Consider the case $p=0$ in Theorem \ref{renlund}.
In our applications, $X_n$ is the fraction of white balls in an urn. If
$W_n$ and $T_n$ denote the number of white balls and the total number
of balls in the urn at time $n$ respectively, then $X_n=W_n/T_n$. 
What is usually  easy to verify is that $T_n=\bigo(n)$, say $T_n\leq Cn$, 
which implies $nX_n\geq \frac 1C W_n$ so that  assumption (\ref{xkrav}) just means
that we need that $W_n\to \infty$. 
\end{rem}

\begin{proof}[Proof of Theorem \ref{renlund}]
We will, for ease of notation, assume in the proof that $p=0$. 
Let $\epsilon>0$ be a number such that $f(x)>0$ if $0<x\leq\epsilon$.

The idea of the proof is to show that should the process 
ever be close to the origin it is very likely that it
doubles its  value before it decreases to a fraction of its
value. So likely in fact, that it will do this time and time
again until it reaches above $\epsilon$.   

Consider the process $\{X_n\}$ after time $N$.
Let $\lambda>0$ be a small constant and let $a\in(0,1-\lambda)$. Define 
\begin{align}\label{stoppisar}  \tau_1 &=\inf\{k\geq N: X_k\geq (2X_N)\wedge \epsilon\} \quad\mbox{and} \nonumber\\
 \hat\tau_1&=\inf\{k\geq N: X_k\leq aX_N\} . \end{align}
Since we assume that $X_n\in(0,1)$ for all $n$, we know
that $X_N>0$ and thus $\hat\tau_1>N$. 
Let $\tau=\tau_1\wedge\hat\tau_1$ and define the two events 
$A=\{\hat\tau_1<\tau_1\}$ and $B=\{\tau_1<\hat\tau_1\}$. 
Anticipating an application of Lemma \ref{zwetc}, we let 
$Z_k=Z_k(N,\tau)$ and $W_m$ as in that lemma.

On the event $A$ we have for any $N<k\leq \hat\tau_1$ that
$X_{k-1}<\epsilon$ and hence
\begin{align*}
  \ex_{k-1}\Delta X_{k-1} &=f(X_{k-1})\ex_{k-1}\gamma_k+\ex_{k-1}\gamma_k U_k >	-\frac{c_u^2K_e}{(k-1)^2}.
\end{align*}
Using this estimate  gives us, on the event $A$,
\begin{align*} 
W_{\tau} &=\sum_{N+1}^{\hat\tau_1}\ex_{k-1}\Delta X_{k-1}-\sum_{N+1}^{\hat\tau_1}\Delta X_{k-1} \\
&\geq -\sum_{k=N+1}^\infty\frac{c_u^2K_e}{(k-1)^2}-(X_{\hat\tau_1}-X_N)\geq X_N(1-a)-\frac{c_u^2K_e}{N-1} \\
&\geq X_N\left(1-a-\frac{c_u^2K_e}{X_N(N-1)}\right)\geq X_N(1-\lambda-a) ,
\end{align*}
where the last step is justified by assumption (\ref{xkrav})
if $X_NN\geq \frac{c_u 2K_e}{\lambda}+1$.

Next, we use assumptions (\ref{ukrav}) and (\ref{fkrav}) to get 
\begin{equation}\label{ykrav}
 \ex_{k-1}(\Delta X_k)^2\leq 2\ex_{k-1}\gamma_k^2[f^2(X_{k-1})+U_k^2]\leq \frac{C_1X_{k-1}}{k^2},
\end{equation}
where $C_1=2c_u^2(K_f'+K_u')$.
This in turn gives, since $X_k< (2X_N)\wedge\epsilon\leq 2X_N$ whenever $k<\tau$,
\begin{align*}
 \ex_N[W_{\tau}^2] 
&\leq \ex_N\left[ \sum_{N+1}^\tau (\Delta X_{k-1})^2\right]
\leq C_12X_N\sum_{N+1}^\infty \frac{1}{k^2} 
\leq C_1\frac{2X_N}{N}. 
\end{align*}

Since $f(x)>0$ on $0<x<\epsilon$ we know from Lemma \ref{convergence} that
$X_n$ eventually must leave $(\zeta,\epsilon)$, for
any $0<\zeta<\epsilon$, and hence that
$B=A^c$. So, we can apply Lemma \ref{zwetc} to get
\begin{align}\label{estimatet}
 \frac{\pr_N(B)}{\pr_N(A)} \geq \frac{\ex_N^2[W_{\tau}|A]}{\ex_N[W_{\tau}^2]}
 \geq \frac{[X_N(1-\lambda-a)]^2N}{C_12X_N}
=\frac{[1-\lambda-a]^2}{2C_1}NX_N.
\end{align}

Exploiting that $\pr(A)+\pr(B)=1$, we see that (\ref{estimatet}) is equivalent to,
with $c_a=(1-\lambda-a)^2/2C_1$,
\begin{equation}\label{sj}
 \pr_N(B)\geq \frac{c_aNX_N}{1+c_aNX_N}=1-\frac{1}{1+c_aNX_N}\geq 1-\frac1{c_aNX_N}.
\end{equation}
Notice that this estimate  
decreases if $a$ increases. 

Next, define stopping times recursively from (\ref{stoppisar})
\begin{align*}
 \tau_{n+1}&=\inf\{k\geq \tau_n: X_k\geq (2X_{\tau_n})\wedge\epsilon\}\quad\mbox{and} \\
\hat\tau_{n+1}&=\inf\{k\geq \tau_n: X_k\leq aX_N\}=\inf\{k\geq \tau_n: X_k\leq a_nX_{\tau_n}\},
\end{align*}
where $a_n$ is some (stochastic) number s.t.\ $a_nX_{\tau_n}=aX_N$ and thus
$a_n\leq a$ since either $X_N\geq \epsilon$ (in which case $a_n=a$) or
\[ X_{\tau_n}\geq (2X_{\tau_{n-1}})\wedge \epsilon \geq (2^n X_N)\wedge \epsilon >X_N \]
(in which case $\tau_n>N$ and $a_n<a$). 
Define the events $A_k=\{\tau_k<\hat\tau_k\}$ and stopping times
$T_k=\tau_k\wedge\hat\tau_k$. Then (\ref{sj}) yields, if $X_{\tau_k}\geq 2^kX_N$,
\[ 
\pr(A_{k+1}|A_k,\fa_{T_k})\geq 1-\frac1{c_{a_k}\tau_kX_{\tau_k}}
\geq 1-\frac1{c_a2^kNX_N}, \]
since $\tau_k\geq N$ and $a_{k}\leq a$.  If $X_{\tau_k}\geq \epsilon$ then
$\pr(A_{k+1}|A_k,\fa_{T_k})=1$. In either case
\[ \pr(A_{k+1}|A_k,\fa_{T_k})\geq 1-\frac1{c_a2^kNX_N}, \] holds.

Now, $\cap_k A_k$ is a subset of the event that the process after $N$ reaches
above $\epsilon$. Hence
\begin{align*}
 \pr_N\bigg(\sup_{j \geq N} X_j &\geq \epsilon\bigg) \geq 
\pr_N\left( \bigcap_{k=1}^\infty A_k\right)
\geq \prod_{k=1}^\infty\left(1-\frac1{c_a2^{k-1}NX_N}\right) \\
&\geq 1-\sum_{k=1}^\infty \frac 1{c_aNX_N2^{k-1}}=1-\frac{2}{c_aNX_N}\to 1,\quad\mbox{as $N\to\infty$}. 
\end{align*}

This contradicts the assumption that $P\{X_n\to 0\}>0$ since this requires that
there is a positive probability that 
for every prescribed $\delta>0$ there is an $N_\delta$ such that 
$n\geq N_\delta$ implies $X_n<\delta$.  
\end{proof}

\subsection{Convergence}
Now we know when we may exclude some unstable points from the
set of limit points. Next, we need to
check that stability of a point $p$ is in fact enough to ensure positive
probability of convergence to $p$. After that we also need to know what happens
at a touchpoint. A touchpoint $p'$ may be thought of as having a ``stable side''
and an ``unstable side''. Intuitively, one may think that convergence to 
$p'$ might be possible from the stable side, which is indeed the case.

For the results of the sections to follow we need the notion of 
\emph{attainability}, recall Definition \ref{attain}.
This is just to rule out trivialities, as there might exists a stable point
in a neighborhood where the process is somehow forbidden to go.
Consider e.g.\ the urn model studied in \cite{HLS80}; an urn has balls 
of two colors, white and black say, and at each timepoint $n$ there is 
a proportion $X_n$ of white balls and a ball is drawn and
replaced along with one additional ball of the same color. 
The probability of drawing a white ball is not $X_n$ but
$h(X_n)$ where $h:[0,1]\to[0,1]$.  This yields a drift function of
$f(x)=h(x)-x$. Consider e.g.\
$h(x)=0$ if $0\leq x\leq 1/2$ and define $h$ on $(1/2,1]$
in such a way that $h>0$ and a stable zero $p$ of $f$ exists there.  
Then the attainability of neighborhoods of this $p$ depends on
the initial condition $X_0$. If
$X_0\leq 1/2$ then $p$ can not be reached as $X_n$ (strictly) decreases 
to zero.

\subsubsection{Stable points}
That convergence to stable points is possible is known in related
models, e.g.\ \cite{HLS80} has a similar result as Theorem \ref{stable}
below. For related multidimensional results, see section 7.1 of
\cite{Ben99}. 

\begin{theorem}\label{stable}
 Suppose $p\in Q_f$ is stable, i.e.\ $f(x)(x-p)<0$ whenever
$x\neq p$ is close to $p$. If every
neighborhood of $p$ is attainable then $\pr(X_n\to p)>0$. 
\end{theorem}

\begin{proof}
\noindent\textbf{Case 1: $p\in(0,1)$.} \\ We can find $a$ and $b$ such that
$a<p<b$ and  $f>0$ on $(a,p)$ and 
$f<0$ on $(p,b)$. Let 
\[ \delta=\min\{b-p,p-a\}\quad \mbox{and}\quad \epsilon=\delta/2.\] 
Define $A_j=\{p-\epsilon\leq X_j<p+\epsilon\}$.
and let $N$ be large
enough so that $\frac{C}{\delta^2(N-1)}\leq \frac 1{12}$,
where $C=K_\Delta^2+2K_ec_u^2$. 

For $k\geq n$, define $Y_k=(X_k-p)^2$.
By attainability there exists an $n\geq N$ such that
$\pr\{A_n\}>0$. 
Define \[ \tau=\inf\{k>n:X_k\leq a\;\mbox{or}\;X_k\geq b\}. \]
We want to show that $\pr(\tau=\infty|A_n)$ is non-zero.

Notice that on $A_n$ we have $Y_n\leq \epsilon^2$ and if $\tau<\infty$ then $Y_\tau\geq \delta^2$.

On $A_n$, for any $n<k\leq\tau$, we have $f(X_{k-1})(X_{k-1}-p)<0$, so that
\begin{align*}
 \ex_nY_k&=\ex_n(X_{k-1}-p+\Delta X_{k-1})^2 \\
&=\ex_n(X_{k-1}-p)^2+\ex_n(\Delta X_{k-1})^2 \\
&\quad\quad +\ex_n[2\gamma_k(X_{k-1}-p)f(X_{k-1})]+2\ex_n(X_{k-1}-p)\ex_{k-1}\gamma_kU_k \\
&\leq \ex_n Y_{k-1}+\frac{K_\Delta^2}{k^2}+\frac{2K_ec_u^2}{(k-1)^2}\\
&\leq \ex_n Y_{k-1}+C(k-1)^{-2}.
\end{align*}
Expanding the above recursion gives a bound on the conditional expectation
\begin{align}
\ex(Y_\tau|A_n)&\leq \ex(Y_n|A_n)+\sum_{k=n+1}^\tau \frac{C}{(k-1)^2} \nonumber \\
&\leq \ex(Y_n|A_n)+\frac{C}{n-1}\leq \epsilon^2+\frac{C}{n-1}. \label{con1} 
\end{align}
Now,
\begin{align*}
 \ex(Y_\tau|A_n)\geq \ex_n(Y_\tau\ind{\tau<\infty}|A_n)\geq \delta^2\pr(\tau<\infty|A_n),
\end{align*}
and this fact in combination with (\ref{con1}) yields
\[ \pr(\tau<\infty|A_n)\leq \left(\frac{\epsilon}{\delta}\right)^2+\frac{C}{\delta^2(n-1)}\leq 
\frac14+\frac1{12}=\frac13. \]
Hence, $\pr(\tau=\infty|A_n)\geq 2/3$ so there is a positive probability
that  $\{X_{n+k}\}$ never leaves $(a,b)$. On the event $\{\tau=\infty\}$
Lemma \ref{convergence} implies that $\lim X_n\in \{a,p,b\}$.
Since we can repeat our argument
with any $a'\in(a,p)$ instead of $a$, and any $b'\in(p,b)$ instead of $b$, this
implies that $\lim X_n=p$ (on the event $\{\tau=\infty\}$).

\textbf{Case 2: $p\in\{0,1\}$} \\
We will prove the statement for $p=0$, with $p=1$ being analogous.
Assume  that $f<0$ on $(0,\delta_*]$, for some $\delta_*>0$. 
Set $N$ so large that $4C/\delta_*^2(N-1)\leq 1/12$, 
$\epsilon_*=\delta_*/4$, 
$A_j=\{X_j\leq \epsilon_*\}$ 
and 
$\tau=\inf\{k>n:X_k\geq \delta_*/2\}$, where $n\geq N$ is such that $\pr(A_n)>0$.
Analogous to Case 1 we calculate 
$\ex_n(X_\tau^2|A_n)\leq \epsilon_*^2+C/(n-1)$ and
$\ex_n(X_\tau^2|A_n)\geq \frac{\delta_*^2}{4}\pr_n(\tau<\infty|A_n)$,
so that $\pr_n(\tau=\infty|A_n)\geq 2/3$. We know that $\pr(A_n)>0$
for some $n\geq N$ by attainability.

So, there is a positive probability of the event 
$B=\{X_{n+k}\leq\delta_*/2$ for all $k\}$. 
By Lemma \ref{convergence} it follows that on this event $B$
we must have $\lim X_n\in\{0,\delta^*/2\}$. But we may repeat the argument
above, choosing any $\delta_*'\in(0,\delta_*)$ in place of $\delta_*$,
concluding that $\lim X_n\in\{0,\delta_*'/2\}$. This makes it clear that
 given the event $B$ we must have $\lim X_n=0$.

\textbf{Postponed proof of Lemma \ref{fbra}}\\ \label{case2}
We will prove Lemma \ref{fbra} in the case when the origin is
attainable, the case of the other boundary point, $x=1$, is analogous.
We assume that $f$ is continuous at $x=0$ and we need to prove that $f(0)\geq 0$. 
Assume the contrary, i.e.\ that $f(0)<0$ and hence, by continuity,
that $f<0$ on $[0,\delta_*)$, for some $\delta_*>0$.

Recall  the notation of Lemma \ref{wn} and
\ref{convergence}; $W_n=\sum_1^n \gamma_kU_k$. Lemma \ref{wn} ensures  that
$\{W_n\}$ converges. Hence, for some large (stochastic) $N_W\geq n$ we have that $i,j\geq N_W$
implies $|W_i-W_j|<\delta_*/2$. 

From identical calculations as in Case 2 above, we can conclude that there is
a positive probability of $\{X_{n+k}<\delta_*/2$
for all $k\}$.
Then 
\[ X_{N_W+k}< X_{N_W}+\sum_{j=N_W+1}^{N_W+k}\gamma_jf(X_{j-1})+\delta_*/2\to-\infty,\;\;\mbox{as $k\to\infty$,} \]
with positive probability, which is a contradiction. Hence, $f(0)\geq 0$. An analogous
argument shows that $f(1)\leq 0$, and Lemma \ref{fbra} follows.
\end{proof}

\subsubsection{Touchpoints}
Theorem \ref{Pem2} below asserts that as long as the slope toward a touchpoint $p$
(from the stable side) is not to steep, convergence is possible. $p$ need in fact not
be a touchpoint as the proof only shows that convergence to $p$ may happen
from the stable side. In our applications, the drift function is differentiable
and thus the slope tends to zero, making the result applicable.

The method of proof is taken from a similar result of \cite{Pem91}, 
which deals with the same urn model as \cite{HLS80}. 
 The interested reader is adviced to read this
article for more, and stronger, results on touchpoints, albeit not
in this more general setting of stochastic approximation. 

\begin{theorem}\label{Pem2}
Suppose  that $p$ is such
that $K(p-x)<f(x)<0$ for some $K<\frac 1{2c_u}$ 
whenever $x>p$ is close to $p$.

Also, assume the following technical condition:
\begin{itemize}
\item[$\bigstar$] Suppose there exists some $p'>p$ such that for every $N\geq 0$ and 
every $y\in(p,p')$ there
exists an $n\geq N$ such that $\pr(X_n>y$ and $X_{n+1}<y)>0$.
\end{itemize}
[Or similarly suppose that $0<f(x)<K(p-x)$ for some $K<\frac 1{2c_u}$
whenever $x<p$ is close to $p$ and assume the existence of 
a $p'<p$ such that for every $N\geq 0$ and every $y\in(p',p)$ there exists 
some $n\geq N$ such that $\pr(X_n<y$ and $X_{n+1}>y)$.]
\end{theorem}
\begin{rem}
 Condition $\bigstar$  states that every point in some neighborhood to the right -- the 
\emph{stable side} -- of $p$ can potentially be down-crossed at some ``later'' time.
\end{rem}

\begin{proof}
First, without loss of generality we make the global assumption that $f(x)<0$ for 
$x\in(0,1]\backslash \{p\}$ (remember Remark \ref{local}). The reason that the origin is not 
included in the interval where  $f$ is negative is Lemma \ref{fbra}. 
These global assumptions are somewhat superfluous, as we will only be concerned
with the behavior of the process to the right of $p$. We will however 
assume that the inequality $K(p-x)<f(x)<0$ holds for all $x>p$.

The idea here is to show that 
\begin{equation}\label{need} 
\pr\{\exists N:n\geq N\;\,\mbox{implies\;\,} X_n>p\}>0, 
\end{equation}
 i.e.\
that it might happen that the process never again reaches below $p$. Given the
event that the process stays above $p$, Lemma \ref{convergence} implies that 
the process must converge to $p$ (from above).

The proof is rather technical and there are numerous constants that 
needs fine tuning in order for everything to work.  First, we will use
a sequence of times $0<T_n<T_{n+1}\nearrow\infty$ and a 
sequence of points $1>p_n>p_{n+1}\searrow p$ starting with an index $N$
large enough so that $p_N<p'$, where $p'$ is defined by condition $\bigstar$. 

We define 
\begin{align*}
\tau_N  &=\inf\{j>T_N: X_{j}<p_N<X_{j-1}\}\quad\mbox{and for $n\geq N$} \\
\tau_{n+1} &=\inf\{j\geq \tau_n:X_j<p_{n+1}\}.
\end{align*}
Notice that 
by $\bigstar$ we have $\pr(\tau_N<\infty)>0$, and 
if 
$X_l\leq p$ for some $l>\tau_N$ then all stopping times are
bounded, namely $\tau_n\leq l$, for all $n\geq N$. 

If we can show that $\pr(\tau_n>T_n,$ for all $n\geq N)>0$,
this will imply (\ref{need}). 

For reasons that only become apparent later
we set 
\begin{align}
 T_n=\exp\left\{\frac{n(1-r)}{\gamma K_1}\right\}\quad\mbox{and}\quad
p_n=p+r^n
\end{align}
where $c_uK<K_1<1/2$ and $\gamma>1$ such that $\gamma K_1<1/2$
and $r\in(0,1)$ is to be specified by the demand that
\begin{equation}\label{defrel} T_n\cdot r^{2n}>1,\quad\mbox{i.e.}\quad 
 \left(r\exp\left\{\frac{1-r}{2\gamma K_1}\right\}\right)^{2n}>1. \end{equation}
If we let $g(r)=re^{(1-r)/2\gamma K_1}$, then
$g(1)=1$ and $g'(1)=1-1/2\gamma K_1<0$, so that we know
that there exists an $r\in(0,1)$ such that $g(r)>1$.
From now on we fix\footnote{We may assume that $r\notin\{1-\gamma K_1\ln m/n: m,n\in\mathbb N\}$
so that $T_n\notin\mathbb N$, as it is easier to consistently think of $T_n$ as a non-integer.}
 such an $r$.
Let 
\begin{equation*}
 A_n=\{\tau_n>T_n\}\quad\mbox{and}\quad B_n=\left\{\sup_{j>\tau_n}X_j\leq p_n+q_n\right\},
\end{equation*}
where $q_n=r^n(\gamma-1)>0$.

Set
\[ Z_k=\ex_{k-1}\Delta X_{k-1}-\Delta X_{k-1}\quad\mbox{and for $m>n$}\quad
 W_{n,m}=\sum_{k=n+1}^m Z_k. \]
We always have the estimate, due to assuming $f\leq 0$,
\[ \ex_{k-1}\Delta X_{k-1}\leq f(X_{k-1})\frac{c_l}{k}+\frac{c_u^2K_e}{(k-1)^2}\leq \frac{c_1}{(k-1)^2}, \]
where $c_1=c_u^2K_e$, and  hence on $A_n$, for $j>\tau_n$,
\[ W_{\tau_n,j}=\sum_{k=\tau_n+1}^j\ex_{k-1}\Delta X_{k-1}-(X_j-X_{\tau_n})
\leq p_n-X_j+\frac{c_1}{\tau_n-1}\leq p_n-X_j+\frac{c_1}{\floor{T_n}}. \]
We will begin by bounding
\begin{align}\label{label72}
 \pr(B_n^c|A_n) &=\pr\left\{\sup_{j>\tau_n}X_j>p_n+q_n\,\big|A_n\right\} \nonumber\\
&\leq \pr\left\{\sup_{j>\tau_n}\left(p_n+\frac{c_1}{\floor{T_n}}-W_{\tau_n,j}\right)>p_n+q_n\,\big|A_n\right\} \nonumber\\
&= \pr\left\{\inf_{j>\tau_n}W_{\tau_n,j}<-q_n+\frac{c_1}{\floor{T_n}}\,\big|A_n\right\}. 
\end{align}
Since $1/\floor{T_n}\leq 2/T_{n}<2r^{2n}$ and 
$q_n=r^n(\gamma-1)$ means that we can make $n$ large enough to ensure that 
\[h_1(n)=-q_n+c_1/\floor{T_n}<0.\] 
Set $\hat\tau=\inf\{k>\tau_n:W_{\tau_n,k} < h_1(n)\}$ and combine the facts
that 
\begin{align*} 
\ex [W_{\tau_n,\infty}^2|A_n] &\leq K_\Delta^2/T_n\quad\mbox{and}\\
\ex [W_{\tau_n,\infty}^2|A_n] &\geq \ex [W_{\tau_n,\hat\tau}^2\ind{\{\hat\tau<\infty\}}|A_n] 
\geq h_1^2(n)\pr(\hat\tau<\infty|A_n)
\end{align*}
so that we can continue the estimates of (\ref{label72})
\begin{equation} \label{howto}
\pr(B_n^c|A_n)\leq \pr(\hat\tau<\infty|A_n)\leq \frac{K_\Delta^2}{T_nh_1^2(n)}. 
\end{equation}
Notice that on the event $B_n$, meaning that $j\geq \tau_n$
implies $X_j\leq p_n+q_n$, we have
\begin{align*}
 \sum_{j=\tau_{n}+1}^{\tau_{n+1}}\ex_{j-1}\Delta X_{j-1}&=
\sum_{j=\tau_{n}+1}^{\tau_{n+1}} \ex_{j-1}\gamma_{j}(f(X_{j-1})+U_{j}) \\
&\geq -\sum_{j=\tau_{n}+1}^{\tau_{n+1}} \frac{c_u}{j}K(X_{j-1}-p)
 -\sum_{j=\tau_{n}+1}^{\tau_{n+1}} |\ex_{j-1}\gamma_{j}U_{j}| \\
&\geq -c_uK(p_n+q_n-p)\sum_{j=\tau_{n}+1}^{\tau_{n+1}} \frac1{j}
 -K_ec_u^2\sum_{j>\tau_n}\frac1{(j-1)^2}. 
\end{align*}
Introduce $\kappa_n=c_uK(p_n+q_n-p)=c_uK\gamma r^n$. 
By the previous bound, given $A_n=\{\tau_n>T_n\}$, the events $A_{n+1}^c=\{\tau_{n+1}\leq T_{n+1}\}$
and $B_n$ together imply, first
\begin{align*}
\sum_{j=\tau_{n}+1}^{\tau_{n+1}}\ex_{j-1}\Delta X_{j-1}
&\geq -c_uK\gamma r^n\sum_{T_n+1<j<T_{n+1}+1}\frac{1}{j}-\frac{K_ec_u^2}{\floor{T_n}} \\
&\stackrel{[1]}{\geq} -c_uK\gamma r^n(\ln T_{n+1}-\ln T_n)-\frac{\kappa_n}{T_n}-\frac{K_ec_u^2}{\floor{T_n}}\\
&\stackrel{[2]}{\geq} -c_ur^n(1-r)K/K_1-\frac{2K_ec_u^2}{\floor{T_n}},
\end{align*}
where $[1]$ is motivated by the fact that if 
 $a,b\in\real\backslash\mathbb N$ are such that $0<a$ and $b>a+1$, then
$\displaystyle \sum_{a<j<b}\frac{1}{j+1}\leq \ln b-\ln a+1/a$.
$[2]$ is motivated by having $n$ large enough since
$\kappa_n$ tends to 0 as $n$ grows.
Secondly, 
by setting $\delta_n=p_n-X_{\tau_n}\leq X_{\tau_n-1}-X_{\tau_n}\leq K_\Delta/\tau_n$,
\begin{align*}
 W_{\tau_n,\tau_{n+1}} &=\sum_{j=\tau_n+1}^{\tau_{n+1}}\ex_{j-1}\Delta X_{j-1} -( X_{\tau_{n+1}}-X_{\tau_n}) \\
&\geq -c_ur^n(1-r)K/K_1-\frac{2K_ec_u^2}{\floor{T_n}}-p_{n+1} + p_{n}-\delta_n \\
&\geq r^n(1-r)[1-c_u^2K/K_1]-c_2/\floor{T_n}=h_2(n),
\end{align*}
 where $c_2=2c_u^2K_e+K_\Delta$. Notice that 
\[ h_2(n)=r^nc_3-c_2/\floor{T_n}\geq r^nc_3-c_22r^{2n},\] with $c_3=(1-r)(1-c_uK/K_1)>0$ and
$r<1$, so
if $n$ is large enough $h_2(n)$ is positive. 

We have just shown that on the event $A_n$ we have
\[ B_n\cap A_{n+1}^c\subseteq \{W_{\tau_n,\tau_{n+1}}\geq h_2(n)\}.\] If we let
$\varsigma=\inf\{m\geq\tau_n:W_{\tau_n,m}\geq h_2(n)\}$ then we also have
$A_{n+1}^c\cap B\subseteq \{\varsigma<\infty\}$ on $A_n$. 
An upper bound on $\pr(\varsigma<\infty|A_n)$ can be calculated analogously
to the bound on $\pr(\hat\tau<\infty|A_n)$ in (\ref{howto}).
This yields an upper 
bound on $\pr(A_{n+1}^c|A_n)$ given by
\begin{align*}
 \pr(A_{n+1}^c|A_n) &=\pr(A_{n+1}^c\cap B_n^c|A_n)+\pr(A_{n+1}^c\cap B_n|A_n) \\
&\leq \pr(B_n^c|A_n)+\pr(\varsigma_n<\infty|A_n) \\
&\leq \frac{K_\Delta^2}{T_nh_1^2(n)}+\frac{K_\Delta^2}{T_nh_2^2(n)} \\
&=\frac{K_\Delta^2}{T_nr^{2n}}\left(\frac{1}{i_1^2(n)}+\frac{1}{i_2^2(n)}\right),
\end{align*}
where
\begin{align*}
 i_1^2(n)&=\left(\gamma-1-\frac{c_1}{r^n\floor{T_n}} \right)^2\to(\gamma-1)^2, \\
 i_2^2(n)&=\left(c_3-\frac{c_2}{r^n\floor{T_n}} \right)^2\to c_3^2, \\
\end{align*}
as $n\to\infty$ since $\frac{1}{r^n\floor{T_n}}<r^n2r^{2n}\to 0$. Thus we can get the  bound
\[ \pr(A_{n+1}^c|A_n)\leq \frac{C}{[g(r)]^{2n}}, \]
for some constant $0<C<\infty$.  So,
\begin{align*}
 \pr(\tau_n>T_n,\;\forall n\geq N)&=\pr(\tau_N<\infty)\prod_{n\geq N}[1-\pr(A_{n+1}^c|A_n)]>0, 
\end{align*}
since the product converges as 
\[ \sum_{n\geq N}\pr(A_{n+1}^c|A_n)\leq \frac{C[g(r)]^2}{[g(r)]^{2N}(g(r)^2-1)}<\infty.\]
\end{proof}

\section{Generalized Pólya urns}
Now, we will apply these result to determine the limiting fraction
of balls in two related urn models. The stochastic approximation
machinery makes this fairly easy albeit hard work since the calculations
to verify the required properties can be rather lengthy.

\subsection{Evolution by one draw}
We now return to the model defined in Section \ref{1-draget}. 
An urn has $W_n$ white and $B_n$ black balls after the $n$'th draw. 
Each draw consists of drawing one ball uniformly from the 
contents of the urn, noticing the color and replacing it along with
additional balls according to the replacement matrix
\begin{align}\label{rep}  & \quad\; \text{W}  \,\;\; \text{B} \nonumber \\
\begin{array}{c}
    \text{W} \\ \text{B} 
   \end{array} &
\left( \begin{array}{c c}
        a & b \\ c & d 
       \end{array} \right), \quad
\begin{array}{r}
 \mbox{where}\;\min\{a,b,c,d\}\geq 0\phantom{,} \\
\mbox{and}\;\max\{a,b,c,d\}>0,
\end{array}
\end{align}
so that, e.g.\ a black ball is replaced along with $c$ additional white and $d$ additional black
balls. The initial values $W_0=w_0>0$ and
$B_0=b_0>0$ are considered fixed, although this makes
no difference to the distribution of the limiting fraction
of white balls, except when $a=d$ and $b=c=0$ as we will see later.

$\wh$ and $\bl$ denote the indicators of getting a white and black ball in draw
$n$, respectively. We define $T_n=W_n+B_n$ and $Z_n=W_n/T_n$ and $Y_{n+1}$
implicitly by $\Delta Z_n=Y_{n+1}/T_{n+1}$, which, after rewriting, gives
\begin{align}\label{yy}
Y_{n+1}=(c+d-a-b)Z_n\wh+[(a-c)\wh-(c+d)Z_n]+c. 
\end{align}
With $Y_{n+1}$ written on this form it is easy to see 
that the drift function 
$f(Z_n)=\ex_n Y_{n+1}$ is given by
\begin{align}\label{ff}
 f(x)=(c+d-a-b)x^2+(a-2c-d)x+c.
\end{align}
By defining $U_{n+1}=Y_{n+1}-f(Z_n)$ and $\gamma_n=1/T_n$ we arrive at the stochastic approximation
representation
\[ \Delta Z_n=\gamma_{n+1}\big[f(Z_n)+U_{n+1}\big]. \]

Clearly,
$f$ and $U_n$ are bounded since $Z_n\in [0,1]$, so that conditions
(ii) and (iii) of Definition \ref{def} are satisfied. 

\vspace{0.3cm}
\noindent\textbf{Condition (i):}\\
Recall that $\gamma_n=1/T_n$. Define
\begin{equation}\label{tminmax} \tmin=\min\{a+b,c+d\}\quad\mbox{and}\quad\tmax=\max\{a+b,c+d\}.\end{equation}
Assume $\tmin>0$, then 
\begin{align*}
 T_n&\leq T_0+n\tmax \Longrightarrow n\gamma_n\geq \frac1{T_0/n+\tmax}\geq \frac1{T_0+\tmax}>0 \\
 T_n&\geq T_0+n\tmin \Longrightarrow n\gamma_n\leq \frac1{T_0/n+\tmin}< \frac1{\tmin}<\infty. 
\end{align*}
Throughout we will assume that $\tmin>0$ and handle the case $\tmin=0$ separately.

\vspace{0.3cm}
\noindent\textbf{Condition (iv):}\\
To verify condition (iv) of Definition \ref{def} we calculate
the expected value of 
\[ \frac{U_{n+1}}{T_{n+1}}=\frac{a-(a+b)Z_n-f(Z_n)}{T_n+a+b}\wh
+\frac{c-(c+d)Z_n-f(Z_n)}{T_n+c+d}\bl.   \]
\begin{align*}
 \ex_n\left[\frac{U_{n+1}}{T_{n+1}}\right] 
&= \frac{a-c+(2c+d-2a-b)Z_n+(a+b-c-d)Z_n^2}{T_n+a+b}Z_n \\
&\quad\quad\quad +\frac{(c-a)Z_n+(a+b-c-d)Z_n^2}{T_n+c+d}(1-Z_n) \\
&=\frac{(a-c)Z_n+(2c+d-2a-b)Z_n^2+(a+b-c-d)Z_n^3}{T_n+a+b} \\
&\quad\quad\quad +\frac{(c-a)Z_n+(2a+b-2c-d)Z_n^2+(c+d-a-b)Z_n^3}{T_n+c+d}\\
&=\big[C_1Z_n+C_2Z_n^2+C_3Z_n^3\big]\frac{c+d-a-b}{(T_n+a+b)(T_n+c+d)}
\end{align*}
for coefficients $C_1=a-c,C_2=2c+d-2a-b$ and $C_3=a+b-c-d$. So, there is a constant $K_e$
(depending on $a,b,c,d$) such that
\[ \left|\ex_n\left[\frac{U_{n+1}}{T_{n+1}}\right]\right|\leq \frac{K_e}{T_n^2}. \]

\vspace{0.3cm}
\noindent\textbf{The error function}\\
In order to apply Theorems \ref{Pem} and \ref{renlund} via 
verification of condition (\ref{v}) and (\ref{ukrav}) we need to calculate what
we will call the  \emph{error function}
\begin{equation}\label{error}
 \e(Z_n) =\ex_n U_{n+1}^2.
\end{equation}
One sees from (\ref{yy}) and (\ref{ff}) that
\begin{equation*}
 U_{n+1}=Y_{n+1}-f(Z_n)=(\wh-Z_n)\Psi(Z_n),
\end{equation*}
where $\Psi(Z_n)=a-c+(c+d-a-b)Z_n$ so that 
\begin{equation*}
 \e(Z_n)=\ex_n[U_{n+1}^2]=Z_n(1-Z_n)[\Psi(Z_n)]^2.
\end{equation*}

\subsubsection{Limit points}
To determine the limits points of the fraction of white balls in this urn model
we know from Theorem \ref{main} that we need to look at zeros of 
\begin{equation}\label{zeros} f(x)=\alpha x^2+\beta x+c, \end{equation}
where $\alpha=c+d-a-b$ and $\beta=a-2c-d$. First notice that 
\begin{equation}\label{01}
 f(0)=c\geq 0\quad\quad\mbox{and}\quad\quad f(1)=-b\leq 0 
\end{equation}
so that, by the continuity and differentiability
of $f$, there must be a point $x^*\in[0,1]$ such that $f(x^*)=0$ and 
$f'(x^*)\leq 0$, see Fig.\ \ref{pic}. A unique zero 
must be the convergence point of the process $\{Z_n\}$ and if more than
one zero exists, we must check which one is stable (if any). 

\begin{figure}[htb] \label{pic}\begin{center}
 \includegraphics[scale=0.5]{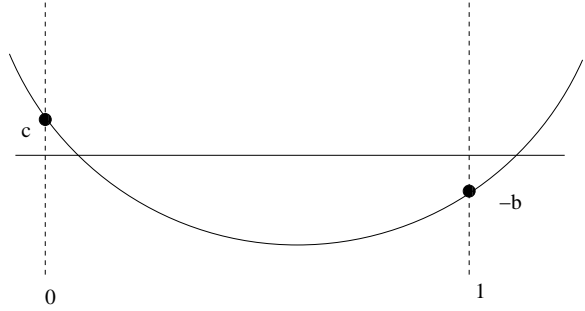}
\end{center}\caption{Schematic picture of $f(x)$ when $\alpha>0$.}\end{figure}

We will look at the possible difficult zeros $x_u$, i.e.\ the ones that
are unstable and where the error terms
are vanishing, in the sense that $\e(x_u)=0$, recall (\ref{error}).

First, $\e(x)\equiv 0$ if and only if $a=c$ and $b=d$. This is not surprising since
there will be no error terms when there is no randomness; this is the urn scheme where $a$
white and $b$ black balls are added whatever color is drawn. The drift function is then
$f(x)=-(a+b)x+a$ so that $x^*=a/(a+b)$ is a unique (stable) zero. 

It follows from (\ref{01}) that if at most two zeros exist\footnote{If there
are more than two zeros then $f$ is identically zero, since $f$ is
a polynomial of order at most 2.} and one of these
is in $(0,1)$ then that one is stable. Hence, an unstable zero, if it exists, must be at
the boundary. By symmetry between colors we need only consider unstable zeros
at the origin. To that end set $c=0$ so that $f(x)=\alpha x^2+\beta x$, with
$\alpha=d-a-b$ and $\beta=a-d$,  has the property $f(0)=0$. In order
for the origin to be unstable we need parameters to make $f(x)\geq 0$ when
$x$ is very small. We need to consider two cases: 

\textbf{(i)} $\beta=0$ but $\alpha\geq 0$. This can only happen if $a=d$ and $b=0$ i.e.\
$f(x)\equiv 0$. Having $f\equiv 0$ makes the sequence $\{Z_n\}$ a
(bounded) martingale and hence a.s.\ convergent. This is in fact the classical 
Pólya-Eggenberger urn model where it is well known
that $Z_n$ converges a.s.\ to a random variable that has a beta distribution
with parameters $w_0/a$ and $b_0/a$, see e.g.\ Theorem 2.2 of \cite{Fre65} or
Theorem 3.2 of \cite{Mah08}.

\textbf{(ii)} $\beta>0$, i.e.\ $a>d$ and $b\geq 0$ arbitrary. The origin is an unstable zero
and not a convergence point. To see why, we need only notice 
that $\e(x)$ certainly can be bounded as (constant)$\cdot x$ and the same is
true of $f(x)$. Considering Remark \ref{notstop}, $W_n\to\infty$ is clear since $a>0$ does imply that
white balls are reinforced infinitely often. Hence, Theorem \ref{renlund} is applicable.

\vspace{0.3cm}
\noindent\textbf{Loose ends}\\
It remains to check what happens if $\tmin=\min\{a+b,c+d\}=0$
(but $a+b+c+d>0$). 
If $c+d=0$ then it is clear that $Z_n\to a/(a+b)$ a.s.\ which is the
unique zero of the driftfunction $f(x)=-(a+b)x^2+ax$. 

The case $a+b=0$ is symmetric.

\vspace{0.3cm}
So we have proved the following.
\begin{theorem}\label{h1}
 Consider the Pólya urn scheme with replacement matrix (\ref{rep}), starting with
a positive number of balls of each color. Then, the limit of the
fraction of balls exists a.s.
Furthermore, apart from the case when $a=d$ and $b=c=0$, in which the
fraction of white balls tends a.s.\ to a beta distribution, the a.s.\
limiting random variable has a one point distribution at $x^*$. This point $x^*$ is a zero  of
(\ref{zeros}) in $[0,1]$ and if two such points exists it has the additional
property that $f'(x^*)<0$.
\end{theorem}

The author does not expect that Theorem \ref{h1} is new (although we have
never seen it written down). In \cite{Gou89} one finds
a similar proposition with less generality as it is demanded that
$a+b=c+d$ although it has the benefit that $a$ and $d$ could be
negative. 

It seems likely that  Theorem \ref{renlund} could be proved using only
the embedding method of Athreya and Karlin into multi-type branching processes, 
see e.g.\ chapter V of \cite{AK68}, but we have not attempted it. 
However, the model in the next section does not fit this embedding method.

In \cite{Plo} a very general extension of the Pólya urn is studied. The urn
may have balls of several colors and balls are drawn with a probability 
according to a function $h$ of the urn content.  At any stage a 
replacement policy is randomly selected from 
a number of different policies, which may include nonbounded
random variables, depending on the colors drawn. However, their
convergence result (Theorem 2.1) is inapplicable to several cases 
in our study due
to their assumption of a unique zero of the resulting drift function. 

Any reader interested in other types of limit theorems for this model 
is advised to consult \cite{AK68}, \cite{Jan04} and \cite{Jan06}.

\subsection{Evolution by two draws}
Again, we will consider an urn with balls of two colors but
now we turn our attention to an urn scheme where two balls are drawn
simultaneously and reinforcement is done according to which of the
three possible combinations of colors this results in. $W_n$ and $B_n$ 
keep their meaning from the previous section but we now assume that $w_0,b_0\geq 2$ so that all
3 combinations of draws have positive probability from the start.

The replacement
matrix becomes 
\begin{align}\label{rep2}  & \quad\; \text{W}  \,\;\; \text{B} \nonumber \\
\begin{array}{c}
    \text{WW} \\ \text{WB} \\ \text{BB}
   \end{array} &
\left( \begin{array}{c c}
        a & b \\ c & d \\ e & f
       \end{array} \right), 
\begin{array}{r}
 \mbox{where}\;\min\{a,b,\ldots,f\}\geq 0\phantom{.} \\
\mbox{and}\;\max\{a,b,\ldots.f\}>0.
\end{array}
 \end{align}
From this we see, e.g.\ that if we draw a white and a black ball these will be replaced
along with additional $c$ white and $d$ black balls.

This model has been studied e.g.\ in Chapter 10 of \cite{Mah08}, where a central limit theorem
for the number of white balls (under parameter constraints, see
remark after Theorem \ref{2drag}) is presented as well as applications 
of the model.

Let
$\ww, \wb$ and $\bb$ be the indicators of the events that draw $n$ results in two white,
one black and one white or two black balls, respectively. Since balls are drawn simultaneously
we have
 \begin{align}\label{forvant} 
\ex_n\ww &= \frac{W_n(W_n-1)}{T_n(T_n-1)} = Z_n^2  -\frac{Z_n(1-Z_n)}{T_n-1}, \nonumber \\
\ex_n\wb &= \frac{2W_nB_n}{T_n(T_n-1)} = 2Z_n(1-Z_n)  + 2\frac{Z_n(1-Z_n)}{T_n-1},  \\
\ex_n\bb &= \frac{B_n(B_n-1)}{T_n(T_n-1)} = (1-Z_n)^2  -\frac{Z_n(1-Z_n)}{T_n-1}. \nonumber
 \end{align}

\begin{rem}
If two balls were drawn with replacement we would have the simpler situation
\[ \ex_n\ww=Z_n^2,\quad\ex_n\wb=2Z_n(1-Z_n)\quad\mbox{and}\quad\ex_n\bb =(1-Z_n)^2.\]
As $T_n\to\infty$ the rightmost parts of (\ref{forvant}) suggest that for large $n$ there is little
difference in sampling the two balls with or without replacement. Sampling without replacement
will make the calculations messier but with the added benefit that is it easy
to see that the result, Theorem \ref{2drag} below, will remain valid 
in the simpler case. In the calculations, terms named ``$R_i$'' 
or ``$R_i(k)$'' are terms that would be zero if we drew with replacement.
\end{rem}

The number of white balls $W_n$ and the total number of balls $T_n$ evolve recursively as
\begin{align*}
 W_{n+1} &= W_n+a\ww+c\wb+e\bb \quad\quad\mbox{and} \\
 T_{n+1} &=T_n+(a+b)\ww+(c+d)\wb+(e+f)\bb.
\end{align*}
Hence, the increments of the fraction of white balls $Z_n$ can be calculated
as
\[ \Delta Z_n= \frac1{T_{n+1}}\big([a-(a+b)Z_n]\ww+[c-(c+d)Z_n]\wb+[e-(e+f)Z_n]\bb\big), \]
which we again denote as $\Delta Z_n=Y_{n+1}/T_{n+1}$. From the above and (\ref{forvant}) we
calculate
\begin{align*}
 \ex_n Y_{n+1}=g(Z_n)+R_n,
\end{align*}
where
\begin{align*}
 g(Z_n)&=\alpha Z_n^3+\beta Z_n^2+\gamma Z_n+e\quad\quad\mbox{and} \\
 R_n &=-\frac{Z_n(1-Z_n)}{T_n-1}\left(a-2c+e+\alpha Z_n\right)
\end{align*}
with
\begin{align}\label{coeff}
 \alpha &=-a-b+2c+2d-e-f, \nonumber \\
 \beta &=a-4c-2d+3e+2f, \quad\mbox{and}  \\
 \gamma &= 2c-3e-f. \nonumber
\end{align}
Setting $U_{n+1}=Y_{n+1}-g(Z_n)$ gives us the stochastic approximation
representation $\Delta Z_n=\frac{1}{T_{n+1}}[g(Z_n)+U_{n+1}]$. 
It is clear that (ii) and (iii) of Definition \ref{def} are satisfied.

\vspace{0.3cm}
\noindent\textbf{Condition (i):}
Define
\begin{equation}\label{tminmax2}
 \tmin=\min\{a+b,c+d,e+f\}\quad\mbox{and}\quad\tmax=\{a+b,c+d,e+f\}.
\end{equation}
Assume $\tmin>0$, then
\begin{align*}
 T_n&\leq T_0+n\tmax \Longrightarrow n\gamma_n\geq \frac1{T_0+\tmax}>0 \\
 T_n&\geq T_0+n\tmin \Longrightarrow n\gamma_n< \frac1{\tmin}<\infty. 
\end{align*}
Throughout we will assume that $\tmin>0$ and handle the case $\tmin=0$ separately.

\vspace{0.3cm}
\noindent\textbf{Condition (iv):} \\
We write the expectation of $U_{n+1}/T_{n+1}$ as 
\begin{align}\label{poly}
 \ex_n\left[\frac{Y_{n+1}-g(Z_n)}{T_{n+1}}\right]=
\frac{p_1(Z_n)}{T_n+a+b}+\frac{p_2(Z_n)}{T_n+c+d}+\frac{p_3(Z_n)}{T_n+e+f},
\end{align}
where each $p_j$ in (\ref{poly}) has the form 
\[ p_j= \sum_{k=0}^5 C_k^{(j)}Z_n^k+R_j(n) \]
with coefficients $C_k^{(j)}$ given by
Table \ref{koeff}. As an example, $p_1$ is calculated from
$\ex_n[a-(a+b)Z_n-g(Z_n)]\ww$.
\begin{table}[htb]\begin{center}\begin{tabular}{|c||c|c|c|} \hline 
$k$ & $C_k^{(1)}$ &$C_k^{(2)}$ &  $C_k^{(3)}$ \\ \hline \hline
0 & 0 & 0 & 0\\ \hline
1 & 0  & $2c-2e$ & $-\gamma-e-f$ \\ \hline
2 & $a-e$  & $-2\gamma-4c-2d+2e$ & $2\gamma+2e+2f-\beta$ \\ \hline
3 & $-\gamma-a-b$ &  $2\gamma+2c+2d-2\beta$ & $2\beta-\alpha-\gamma-e-f$ \\ \hline
4 & $-\beta$ & $2\beta-2\alpha$ & $2\alpha-\beta$ \\ \hline
5 & $-\alpha$ & $2\alpha$ & $-\alpha$ \\ \hline
\end{tabular}\caption{Coefficients of the polynomials $p_j$, $j=1,2,3$.}\label{koeff}\end{center}\end{table}
Each $R_j(n)$ is a polynomial in $Z_n$ divided by $T_n-1$, more precisely 
\begin{align*}\label{rn}
 R_1(n) &=\frac{Z_n(1-Z_n)}{T_n-1}[(e-a)+(\gamma+a+b)Z_n+\beta Z_n^2+\alpha Z_n^3], \nonumber \\
 R_2(n) &=\frac{Z_n(1-Z_n)}{T_n-1}2[(c-e)-(\gamma+c+d)Z_n-\beta Z_n^2-\alpha Z_n^3], \nonumber \\
 R_3(n) &=\frac{Z_n(1-Z_n)}{T_n-1}[(\gamma+e+f)Z_n+\beta Z_n^2+\alpha Z_n^3]. \nonumber
\end{align*}
We want to show that $|\ex_n U_{n+1}/T_{n+1}|=\bigo (T_n^{-2})$. 
The $R_j$ terms clearly satisfy $|R_j(n)/T_{n+1}|=\bigo (T_n^{-2})$. 

Recalling (\ref{coeff}), and plugging  these in, shows that for 
each $k=0,1,\ldots,5$ we have $C_k^{(1)}+C_k^{(2)}+C_k^{(3)}=0$.
This gives us  
\begin{align*}
\mbox{(\ref{poly})}&=\sum_{k=0}^5\left[\frac{C_k^{(1)}}{T_n+a+b}+\frac{C_k^{(2)}}{T_n+c+d}+\frac{C_k^{(3)}}{T_n+e+f}\right]
         Z_n^k+ \mathscr R(n)\\
&=\sum_{k=0}^5\left[\frac{[C_k^{(1)}+C_k^{(2)}+C_k^{(3)}]T_n^2+c_1^{(k)}T_n+c_2^{(k)}}{(T_n+a+b)(T_n+c+d)(T_n+e+f)}\right]
         Z_n^k+ \mathscr R(n)\\
&=\sum_{k=0}^5\left[\frac{c_1^{(k)}T_n+c_2^{(k)}}{T_n^3+c_3^{(k)}T_n^2+c_4^{(k)}T_n+c_5^{(k)}} \right]
         Z_n^k+ \mathscr R(n),
\end{align*}
where $\mathscr R(n)=\frac{R_1(n)}{T_n+a+b}+\frac{R_2(n)}{T_n+c+d}+\frac{R_3(n)}{T_n+e+f}$ and
$c_1^{(k)},\ldots,c_5^{(k)}$ are
some constants  whose exact value is of no importance. 
This makes it clear that \[ |\ex_n U_{n+1}/T_{n+1}|=\bigo(T_n^{-2}).\] 

\vspace{0.3cm}
\noindent\textbf{The error function} \\
In order to apply Theorems \ref{Pem} and \ref{renlund} via 
verification of condition (\ref{v}) and (\ref{ukrav}) we need to 
calculate the second moment of
\begin{align*}
 U_{n+1} = Y_{n+1}-g(Z_n) &= [a-(a+b)Z_n][\ww-Z_n^2] \\ 
 &+[c-(c+d)Z_n][\wb-2Z_n(1-Z_n)] \\
&+[e-(e+f)Z_n][\bb-(1-Z_n)^2].
\end{align*}
Excruciating calculations show that the error function is given by
\[ \e(Z_n)=\ex_n U_{n+1}^2 = Z_n(1-Z_n)\Psi(Z_n)+R_n, \]
where $R_n$ is a polynomial in $Z_n$ divided by $T_n-1$ (so this
term tends to zero for large $n$) and $\Psi(x)$ is a polynomial
of order 4 given by
\begin{align*}
\Psi (x) &=  (a+b-2c-2d+e+f)^2\cdot x^4+ \\
&[-2(a+b-2c-2d+e+f)(a-2c+e) \\ &\quad +(e+f-a-b)^2-4(e+f-c-d)^2]\cdot x^3+ \\  
&[(a-2c+e)^2+2(a-e)(e+f-a-b) \\ &\quad-8(c-e)(e+f-c-d)+2(e+f-c-d)^2]\cdot x^2+ \\
&[(a-e)^2-4(c-e)^2+4(c-e)(e+f-c-d)]\cdot x + \\ &2(c-e)^2,
\end{align*}
which is too complicated a formula to work with. Working through the expression one can arrive at the 
form
\begin{equation}\label{psi} 
\Psi(x)=2x^2(A_x+C_x)^2+x(1-x)B_x^2+2(1-x)^2C_x^2,
\end{equation}
where
\begin{align*}
 A_x &=(-a-b+2c+2d-e-f)x+a-2c+e, \\ B_x&=(e+f-a-b)x+a-e\quad \mbox{and} \\
 C_x&=(e+f-c-d)x+c-e,	 
\end{align*}
and the relation 
\begin{equation}\label{rel}
A_x=B_x-2C_x 
\end{equation}
 holds.

\subsubsection{Limit points}
To determine the limiting fraction of white balls we need to examine
the zeros of 
\begin{equation}\label{zeros2} g(x)=\alpha x^3+\beta x^2+\gamma x +e,\end{equation} where
$\alpha=-a-b+2c+2d-e-f$, $\beta=a-4c-2d+3e+2f$ and
$\gamma=2c-3e-f$. We see that
\begin{equation}\label{endps}
 g(0)=e\geq 0\quad\quad\mbox{and}\quad\quad\quad\quad g(1)=-b\leq 0.
\end{equation}
By continuity and differentiability there will thus exists a point $x^*$
with $f(x^*)=0$ and $f'(x^*)\leq 0$.  A difference with the previous
model, where the urn evolved by a single draw each time, is that
we now have more types of equilibrium points. Previously, we
only encountered unstable zeros on the boundary, which we resolved
with Theorem \ref{renlund}. Now we will also make use of Theorems 
\ref{Pem} and \ref{Pem2}.

\begin{rem}
 There will be urn schemes with a unique zero, e.g.\ the case 
$(a,b,\ldots,f)=(3,2,2,3,1,4)$, where
\[ g(x)=-3x+1, \]
 and $1/3$ is unique. Another example is given by
$(a,b,\ldots,f)=(9,1,2,3,1,7)$ where $g(x)=-8(x-1/2)^3$.

\end{rem}

\begin{rem} \label{attstar}\textbf{On attainability and condition $\bigstar$ of
Theorem \ref{Pem2}} \\ 
Consider the replacement matrix \ref{rep2}. Let $\hat w_n, \hat m_n$ and 
$\hat b_n$ denote the number of times that draws  up to time $n$ 
has resulted in the combinations $WW, WB$ and
$BB$, respectively.

Then
\[ Z_n=\frac{W_0+a\hat w_n+c\hat m_n+e\hat b_n }{T_0+(a+b)\hat w_n
 +(c+d)\hat m_n+(e+f)b_n)}. \]
Since \[ \pr(\hat w_n=i, \hat m_n=j,\hat b_n=k)>0 \] for any combination
$0\leq i,j,k$ such that $i+j+k=n$, it follows that any open set in
$[L,U]$ is attainable, where 
\[ L=\min\left\{\frac{a}{a+b},\frac{c}{c+d},\frac{e}{e+f} \right\}\quad\mbox{and}\quad U=\max\left\{\frac{a}{a+b},\frac{c}{c+d},\frac{e}{e+f} \right\}. \]

Hence, if $p_s$ is a stable zero of $f$ and $p\in[L,U]$ then the conditions
of Theorem \ref{stable} is fullfilled and $\pr(Z_n\to p_s)>0$.

We can also see that condition $\bigstar$ of Theorem \ref{Pem2} is 
satisfied if $p_t$ is a touchpoint and $p_t\in(L,U)$. Furthermore,
since the drift is continuously differentiable, the slope will tend to zero
close to $p_t$, making Theorem \ref{Pem2} applicable.

We will not attempt to prove that it will always be the case that
neighborhoods of stable points are attainable and
that condition $\bigstar$ is satisfied whenever there is a touchpoint.
Our attempts to do so yields too messy calculations, but it seems 
reasonable that this is true.

In any specific situation there is no problem in verifying these
conditions. 
\end{rem}

\begin{rem}
 There will be urn schemes where the set of stable zeros contains exactly two
points and with unstable zeros in $(0,1)$.
For example the case
$(a,b,\ldots,f)=(15,3,4,1,3,21)$, where
\[ g(x)=-32x^3+48x^2-22x+3=-32(x-1/4)(x-1/2)(x-3/4),   \]
and $1/2$ is unstable whereas $1/4$ and $3/4$ are stable.
Notice that $L=3/24<1/4$ and $3/4<U=15/18$ so that both
stable points are possible convergence points by Remark \ref{attstar}.
\end{rem}

\begin{rem}
Touchpoints may arise. If $(a,\ldots,f)=(35,9,1,1,3,21)$ then
\[ g(x)=-64x^3+80x^2-28x+3=-64(x-1/4)^2(x-3/4), \]
where $1/4$ is a touchpoint and $3/4$ is stable. Notice
that $L=3/24<1/4<3/4<35/44=U$ so that both the stable point
and the touchpoint are possible convergence points
by Remark \ref{attstar}.
\end{rem}

\vspace{0.3cm}
\noindent\textbf{No unstable equilibrium in $(0,1)$ with vanishing error terms.} \\
Now we will examine whether there could exist an unstable zero $x_u$ in $(0,1)$
such that $\e(x_u)=0$, i.e.\ an unstable zero to which we can not apply
Theorem \ref{Pem}. We recall that 
$\e(x)=x(1-x)\Psi(x)+R_n$ where $R_n=\bigo (T_n^{-1})$ and $\Psi$
is given in (\ref{psi}). Hence, we need to look at points $x\in(0,1)$ such that
$\Psi(x)=0$.

First, $\Psi(x)\equiv 0$ if $A_{x}\equiv B_{x}\equiv C_{x}\equiv 0$. It is easy to calculate
and intuitive that this can only happen if $a=c=d$ and $b=d=f$, since this is the case
when there is no randomness involved in the urn scheme, $a$ white and $b$ black balls
are added whatever is drawn. Then $g(x)=-(a+b)x+a$ so that $x^*=a/(a+b)$ is unique.

Next, we need to solve $\Psi(x)=0$ for $0<x<1$. We will do this by going
through the cases when exactly one of $C_x, B_x$ or $A_x$, or
none, is zero. This suffices due to relation (\ref{rel}).

Note, since the drift function is a polynomial of order at most $3$ with
boundary condition (\ref{endps}), the only time
when an unstable $x^*\in(0,1)$ has $g'(x^*)=0$ is when $g(x)\equiv 0$. 
This case is special and will be treated below. Hence,
 we need only verify that if $x^*\in(0,1)$ is a point where
$\e$ vanishes, then $x^*$ is not strictly unstable, i.e.\ that $g'(x^*)\leq 0$.

\vspace{0.3cm}
\noindent\textbf{The case $C_x\equiv 0$}\\
If $C_x\equiv 0$, i.e.\ $c=e$ and $f=d$, then from (\ref{rel}) we have that 
$A_x=B_x=(e+f-a-b)x+a-e$. Then
\[ \Psi(x)=2x^2A_x^2+x(1-x)A_x^2=x(1+x)A_x^2. \]
We assume $B_x$ is not identically zero,
so if $e+f=a+b$ then $B_x=a-e$ is never zero. 
If $e+f\neq a+b$ then $B_{x^*}=0$ for $x^*=\frac{e-a}{e+f-a-b}=\frac{\N}{\D}$.
The drift function $g(x)$ is now
\begin{align*} g(x) &=(-a-b+e+f)x^3+(a-e)x^2-(e+f)x+e \\ &=\D x^3-\N x^2-(e+f)x+e, \end{align*}
so that $g(x^*)=e-\frac{(e+f)(e-a)}{e+f-a-b}$ and thus
$g(x^*)=0$ if $af=eb$. If $e\neq 0$ then
\begin{align*}
 g'(x^*) &=\frac{(e-a)^2}{e+f-a-af/e}-(e+f) \\
 &=(1-a/e)\frac{e^2}{e+f}-(e+f)\leq 0. 
\end{align*}
If $e=0$, then $af=0$ so that $x^*=0$ or $x^*=a/(a+b)$. In 
the latter case we have $g'(\frac{a}{a+b})=-a^2/(a+b)\leq 0$.

\vspace{0.3cm}
\noindent\textbf{The case $B_x\equiv 0$}\\
If $B_x\equiv 0$ then $a=e$ and $f=b$, $A_x=-2C_x$ and
\[ \Psi(x)=2x^2(-C_x)^2+2(1-x)^2C_x^2= C_x^2(4x^2-4x+2), \] 
where $4x^2-4x+2$ has no roots in $(0,1)$. 
We assume that $A_x=-2C_x$
is not identically zero so if $a+b=c+d$ then
$A_x=2(a-c)$ is never zero. If $a+b\neq c+d$
then $A_{x^*}=0$ for $x^*=\frac{a-c}{a+b-c-d}=\frac{\N}{\D}$. Now
\begin{align*} g(x)&=2(-a-b+c+d)x^3+2(2a+b-2c-d)x^2+(2c-3a-b)x+a\\
 &=-2\D x^3+(2\N+2\D)x^2-(2\N+a+b)x+a,
\end{align*}
so that $g(x^*)=a-\frac{(a+b)(a-c)}{a+b-c-d}=0$
if $cb=ad$. If $a\neq 0$ then
\begin{align*} 
g'(x^*) &= -2\frac{(a-c)^2}{a+b-c-cb/a}+a-2c-b \\
&=-2(1-c/a)\frac{a^2}{a+b}+a-b-2c=-\frac{a^2+b^2+2bc}{a+b}\leq 0.
\end{align*}
If $a=0$ then $cb=0$ so $x^*=0$ or $x^*=c/(c+d)$. In the latter
case we have $g'(x^*)=-2cd/(c+d)\leq 0$.

\vspace{0.3cm}
\noindent\textbf{The case $A_x\equiv 0$}\\
If $A_x\equiv 0$ then $2c=a+e$, $2d=b+f$, $B_x=2C_x$ and
\[ \Psi(x)=2x^2C_x^2+x(1-x)(2C_x)^2+2(1-x)^2C_x^2=2C_x^2.\]
  We assume that
$B_x=2C_x$ is not identically zero so if $e+f-a-b=0$ then
$B_x=a-e$ is never zero. If $e+f\neq a+b$ then
$B_{x^*}=0$ for $x^*=\frac{e-a}{e+f-a-b}
\frac{\N}{\D}$ and
\begin{align*}
 g(x)&=(e+f-a-b)x^2+(a-2e-f)x+e \\
&=\D x^2-(\N+e+f)x+e,
\end{align*}
so that \[ g'(x^*)=\N-e-f=-a-f\leq 0.\] 

\vspace{0.3cm}
\noindent\textbf{The case $A_x,B_x,C_x$ not $\equiv 0$}\\
Suppose $a\neq e\neq c$, $a+b\neq e+f\neq c+d$, $2c\neq a+c$ and
$a+b+c+d\neq 2c+2d$. The only chance of having $\Psi(x)=0$, for $x\in(0,1)$,
is for $A_x, B_x$ and $C_x$ to be zero simultaneously.
A common zero $x^*$ of $A_x, B_x$ and $C_x$ when none of  these is
identically zero imposes 
\[ x^*=\frac{e-a}{e+f-a-b}=\frac{e-c}{e+f-c-d}=\frac{a-2c+e}{a+b-2c-2d+e+f} \]
which is the case whenever $e(b-d)+c(f-b)+a(d-f)=0$. 

Setting $x^*=(e-a)/(e+f-a-b)$ and $d=[b(e-c)+f(c-a)]/(e-a)$ and using
Maple yields the simple expression
\[ g(x^*)=\frac{af-eb}{e+f-a-b}, \]
which is zero if $af=eb$. The derivative simplifies to
\[ g'(x^*)=-a\frac{e-a}{e-a+f-b}-(f+2c)\frac{f-b}{f-b+e-a}. \]
If $a=0$ then $be=0$ so $g'(x^*)$ is $-(f+2c)$ or $-(f+2c)\frac{f}{f+e}$.
If $a\neq 0$  we can write $f=be/a$ and 
\[ g'(x^*)=-a\frac{a}{a+b}-(f+2c)\frac{b}{a+b}\leq 0. \]

So we can determine that there is no strictly unstable zero of $g(x)$ in $(0,1)$ 
such that the error terms are vanishing. 

\vspace{0.3cm}
\noindent\textbf{Unstable boundary points}\\
Next, we check the boundary. By symmetry of colors we need only consider
an unstable zero of $g$ at the origin. To that end set 
$e=0$ so that $g(x)=\alpha x^3+\beta x^2+\gamma x$ has $g(0)=0$. We check the cases 
where $g\geq 0$ close to the origin: \\
\textbf{(a)} $\gamma=0$, $\beta=0$ and $\alpha\geq 0$ is only possible
if $2c=f$, $2d=a$ and $b=0$ so that $\alpha=0$ and hence
$g(x)\equiv 0$. It is then, in some sense, the 2-draw version
of the classical Pólya urn. The special case of 
$2c=2d=a=f$ has been studied in \cite{MC05} and they show that 
the limiting variable has an absolutely continuous distribution.
They also include a simulation study that indicates that the 
limiting distribution ``resembles'' the beta distribution. 

By Theorem \ref{Pem}, 
we may only conclude 
that the limiting distribution has no point masses on $(0,1)$. \\
\textbf{(b)} $\gamma=0$ and $\beta>0$, i.e.\ $2c=f$ and $a>2d$.  \\
\textbf{(c)} $\gamma>0$, i.e.\ $2c>f$. \\ In both (b) and (c)
the conditions of Theorem \ref{renlund} are satisfied. $kZ_k$ can
be made arbitrarily big since $W_k\to\infty$. In the second case this
is due to $c>0$. Since the combination $WB$ will always be drawn
infinitely often, it follows that white balls will tend to infinity.
In the first case, either $c>0$ and we are done, or $c=f=e=0$
and $a>0$ so that the combination $WW$ will be drawn infinitely
often.

Also, both $g(x)$ and $\e(x)$ behave like 
(constant)$\cdot x$ when $x$ is close to zero. Thus, convergence
to a strictly unstable boundary point is impossible.

\vspace{0.3cm}
\noindent\textbf{Loose ends}\\
Here we examine what happens if $\tmin=\min\{a+b,c+d,e+f\}=0$.

\vspace{0.1cm}
\noindent\textbf{1.}
$c=d=e=f=0$ has drift $g(x)=-(a+b)x^3+ax^2$. It is clear that
$Z_n$ does converge to $a/(a+b)$ a.s.  (i.e.\ the stable zero of $g$)
since any draw that alters the urns composition does so by increasing
the number of white balls by $a$ and the numbers of black balls by $b$.

\vspace{0.1cm}
\noindent\textbf{2.} 
$a=b=c=d=0$ is symmetric to the above case $c=d=e=f=0$; $Z_n$ will
converge to the stable zero of $g$.

\vspace{0.1cm}
\noindent\textbf{3.}
$a=b=e=f=0$ has \[g(x)=2(c+d)x^3-(4c+2d)x^2+2cx=2x(1-x)[c-(c+d)x]\]
with $c/(c+d)$ being the only stable zero. It is clear that 
$Z_n$ converges to this point a.s.

\vspace{0.1cm}
\noindent\textbf{4.}
$a=b=0$ and $\min\{c+d,e+f\}>0$  so that nothing happens when 
two white balls
are drawn (except that they are replaced in the urn). 
Let $\tau_1=\inf\{k\geq 1: T_k>T_{k-1}\}$ and for $n\geq 1$
\[ \tau_{n+1}=\inf\{k>\tau_n: T_k>T_{\tau_n}\}. \]
By looking at the sequence $Z_{\tau_n}$ we ignore the times when
two white balls are drawn. This makes no difference to the 
limit  since $Z_{\tau_n+k}=Z_{\tau_n}$ for 
$0\leq k<\tau_{n+1}-\tau_n$.

 However, since
\begin{align*}
 \ex_{\tau_n} \hbb_{\tau_{n+1}}&=
 \frac{ 
\frac{B_{\tau_n}(B_{\tau_n}-1)}{T_{\tau_n}(T_{\tau_n}-1)}
}{
\frac{B_{\tau_n}(B_{\tau_n}-1)}{T_{\tau_n}(T_{\tau_n}-1)}+
\frac{2W_{\tau_n}B_{\tau_n}}{T_{\tau_n}(T_{\tau_n}-1)}
}=
\frac{B_{\tau_n}-1}{B_{\tau_n}+2W_{\tau_n}-1}\quad\mbox{and} \\
\ex_{\tau_n}\hwb_{\tau_{n+1}} &=\frac{2W_{\tau_n}}{W_{\tau_n}+2B_{\tau_n}-1},
\end{align*}
it is more convenient to define $\hat T_n=B_{\tau_n}+2W_{\tau_n}-1$ and
$\hat Z_n=2W_{\tau_n}/\hat T_n$. It is straightforward to compute
$\Delta\hat Z_n=\hat Y_{n+1}/\hat T_{n+1}$ where
\[ \hat Y_{n+1}=\hwb_{\tau_{n+1}}(2c-(2c+d)\hat Z_n)+\hbb_{\tau_{n+1}}(2e-(2e+f)\hat Z_n), \]
so that 
\[ \hat g(\hat Z_n) =\ex_{\tau_n} \hat Y_{n+1}=(2e+f-2c-d)\hat Z_n^2+(2c-4e-f)\hat Z_n+2e. \]
We also get the error function, with $\hat U_{n+1} =\hat Y_{n+1}-g(\hat Z_n)$, as
\[ \e(\hat Z_n) =\ex_{\tau_n}(\hat U_{n+1}^2)=[(2e+f-2c-d)\hat Z_n+2c-2e]^2\hat Z_n(1-\hat Z_n). \]
Since we have assumed $\min\{c+d,e+f\}>0$ we know that (i) of Definition (\ref{def})
is satisfied, and one easily verifies (ii)-(iv).

As $g(0)=2e\geq 0$ and $g(1)=-d\leq 0$ any unstable zero of $\hat g$
must be on the boundary $\{0,1\}$. We can apply 
Theorem \ref{renlund} to conclude that $\hat Z_n$ will not converge
to an unstable boundary point.  

So, we have the original process $Z_n=W_n/(W_n+B_n)$ described by the driftfunction
\begin{align*} 
g(x)&=(2c+2c-e-f)x^3+(-4c-2d+3e+2f)x^2+(2c-3e-f)x+e \\
&=(1-x)[(-2c-2d+e+f)x^2+(2c-2e-f)x+e]
 \end{align*}
and our ''new'' process $\hat Z_n=W_{\tau_n}/(2W_{\tau_n}+B_{\tau_n}-1)$
described by
\begin{align*}
 \hat g(x)=(-2c-d+2e+f)x^2+(2c-4e-f)x+2e.
\end{align*}
Now, as $T_n\to\infty$, the limit $x=\lim_n Z_n$ and $y=\lim_n\hat Z_n$
are related as
\[ y=\frac{2x}{x+1}\quad\mbox{and}\quad x=\frac{y}{2-y}. \]
In particular $x=0$ is equivalent to $y=0$ and $x=1$ is equivalent to 
$y=1$.

A straightforward calculations shows that
\begin{equation}\label{ghatg} \frac 12(1+x)^2\hat g\left(\frac{2x}{x+1}\right)=\frac{g(x)}{1-x}, \end{equation}
so that $\hat g$ and $g$ have the same zeros, with the possible exception of $x=1$.

For $x=1$ we examine two cases:
\begin{itemize}
\item[(i)] If $d>0$ then $\hat g(1)=d>0$, whereas $g(1)=0$ always, so that $1\notin \hat Z_\infty$, which implies $1\notin Z_\infty$.
But $g'(1)=d>0$ so that nonconvergence to $1$ is what we would expect..

\item[(ii)] If $d=0$ then $\hat g(1)=g(1)=0$. We examine the behavior of these close to $1$ via the
calculations 
\begin{align*}\hat g(1-\epsilon)&=\epsilon[2e\epsilon+(2c-f)(1-\epsilon)]\quad\mbox{ and }\\
g(1-\epsilon)&=\epsilon^2[e\epsilon+(2c-f)(1-\epsilon)]. 
\end{align*}
\end{itemize}
Hence $x=1$ is stable/unstable simultaneously for $\hat g$ and $g$.

Differentiating (\ref{ghatg}) for $x\neq 1$ yields
\begin{equation*}
 (1+x)\hat g\left(\frac{2x}{x+1}\right)+\hat g'\left(\frac{2x}{x+1}\right)
=\frac{1}{(1-x)^2}g(x)+\frac{1}{1-x}g'(x),
\end{equation*}
so that at any point $x<1$ where $\hat g$ and $g$ vanishes we have
\[ \hat g'\left(\frac{2x}{x+1}\right)=\frac{1}{1-x}g'(x), \]
i.e.\ $\hat g$ and $g$ have the same stable and strictly unstable points
(if any).
For any $x<1$ such that $\hat g'$ and $g'$ vanishes we also have
\[ \hat g''\left(\frac{2x}{x+1}\right)=\frac{1}{1-x}g''(x), \]
so that $\hat g$ and $g$ have identical touchpoints (if any).

Thus the convergence of $\hat Z_n$ to a stable zero of $\hat g$ implies
the convergence of $Z_n$ to a stable zero of $g$. 
Also, the non-convergence to an unstable zero of $\hat g$ implies the
non-convergence to an unstable zero of $g$.

\vspace{0.1cm}
\noindent\textbf{5.}
$e=f=0$ and $\min\{a+b,c+d\}>0$ is symmetric to 4.

\vspace{0.1cm}
\noindent\textbf{6.}
$c=d=0$ and $\min\{a+b,e+f\}>0$. Define $\tau_n$
as in the previous section and $\hat Z_n=Z_{\tau_n}$. Now
\[\ex_{\tau_n} \hbb_{\tau_{n+1}}=\frac{W_{\tau_n}(W_{\tau_n}-1)}{W_{\tau_n}(W_{\tau_n}-1)+B_{\tau_n}(B_{\tau_n}-1)}\] 
is a difficult expression to work with directly,
so we rewrite it as 
\begin{align*}
 \ex_{\tau_n} \hbb_{\tau_{n+1}}&=\frac{\hat Z_n^2}{\hat Z_n^2+(1-\hat Z_n)^2}+\frac{R_n}{\hat T_n}, \quad\mbox{where} \\
R_n&=-\frac{\hat Z_n(1/2-\hat Z_n)(1-\hat Z_n)}{[\hat Z_n^2+(1-\hat Z_n)^2][\hat Z_n^2+(1-\hat Z_n)^2-1/\hat T_n]}.
\end{align*}
Next, set $\hat Y_{n+1} =\hat T_{n+1}\Delta \hat Z_n$. Then
\begin{align*}
 \hat Y_{n+1}=(e+f-a-b)\hat Z_n\hbb_{\tau_{n+1}}+(a-e)\hbb_{\tau_{n+1}}-(e+f)\hat Z_n+e,
\end{align*}
so that 
\begin{align*}
 \ex_{\tau_n}\hat Y_{n+1}&=g(\hat Z_n)+\frac{R_n[\hat\alpha \hat Z_n+(a-e)]}{\hat T_n},\quad\mbox{where} \\
g(x)&=\hat\alpha \frac{x^3}{x^2+(1-x)^2}+(a-e)\frac{x^2}{x^2+(1-x)^2}+(-e-f)x+e,\quad\mbox{and} \\
\hat\alpha&=e+f-a-b.
\end{align*}
Next, 
\[ \hat U_{n+1} =\hat Y_{n+1}-g(\hat Z_n)
 =[\hat\alpha \hat Z_n+(a-e)]\left[\hww_{\tau_{n+1}}-\frac{\hat Z_n^2}{\hat Z_n^2+(1-\hat Z_n)^2}\right],
\]
which yields the error function
\begin{align*}
 \e(\hat Z_n)=\ex_{\tau_n}\hat U_{n+1}^2=[\hat\alpha \hat Z_n&+(a-e)]^2\bigg[
\frac{\hat Z_n^2(1-\hat Z_n)^2}{[\hat Z_n^2+(1-\hat Z_n)^2]^2} \\
&+\frac{R_n[1-2\hat Z_n^2/(\hat Z_n^2+(1-\hat Z_n)^2)]}{\hat T_n}\bigg].
\end{align*}
One may also verify (iv) of Definition (\ref{def}) by calculating:
\begin{align*}
 \ex_{\tau_n}\frac{\hat U_{n+1}}{\hat T_{n+1}}=
-\frac{1}{(\hat T_n+a+b)(\hat T_n+e+f)}\cdot\frac{\hat\alpha[\hat\alpha \hat Z_n+(a-e)]}{[\hat Z_n^2+(1-\hat Z_n)^2]^2}
\end{align*}
which certainly is $\bigo (T_n^{-2})$.

Next we compare $\hat Z_n$ with the original process $Z_n$, with drift
\[ g(x)=(-e-f-a-b)x^3+(a+3e+2f)x^2+(-3e-f)x+e. \]  
It is straightforward to verify that
\begin{equation}\label{ghatg2} [x^2+(1-x)^2]\hat g(x)=g(x), \end{equation}
so that $\hat g$ and $g$ have the same equilibrium points. 
Differentiating (\ref{ghatg2}) yields
\begin{equation*}
 (4x-1)\hat g(x)+[x^2+(1-x)^2]\hat g'(x)=g'(x),
\end{equation*}
so, at any point where $\hat g=g=0$ we have $\sgn [\hat g'(x)]=\sgn [g'(x)]$.
Differentiating again at a point where $\hat g'=g'=0$ yields
$[x^2+(1-x)^2]\hat g''(x)=g''(x)$. In conclusion, $\hat g$ and 
$g$ have ''similar'' equilibrium points in that they are stable,
strictly unstable, or touchpoints simultaneously.

We have proved the following.

\begin{theorem}\label{2drag}
 Suppose that the Pólya urn scheme of drawing two
balls (with or without replacement) according to (\ref{rep2}) has $w_0,b_0>1$
(or just $w_0,b_0>0$ if drawn with replacement). Then, the limit of the fraction
of balls exists a.s.
Furthermore, apart from the case $a=2d$, $f=2c$ and $b=e=0$, in which all we may conclude
is that the limiting variable of the fraction of white balls has no point masses in $(0,1)$,
the limiting random variable has support only on the zeros of $g$,
defined in (\ref{zeros2}), such that the derivative $g'$ is nonpositive there. 
\end{theorem}
In Theorem 10.1 of \cite{Mah08} one can find a central limit theorem for the
number of white balls in the urn scheme \ref{rep2} with the added constraints
of a constant row sum larger than one (i.e.\ $a+b=c+d=e+f=K\geq 1$) and $a-2c+e=0$, 
which together has the effect that the drift function becomes linear (i.e.\ $\alpha=\beta=0$
in \ref{coeff}). It is also noted there (Proposition 10.3) that the
fraction of white balls converge (in probability) to $-e/\gamma$.

\vspace{0.3cm}
\noindent\textbf{Acknowledgements} 

\noindent 
This paper, with minor modifications, has served as my licentiate thesis which was defended
2009-03-06 at Uppsala university.
I am indebted to my PhD supervisors Sven Erick Alm and Svante Janson for their
encouragement and for providing me with numerous ways to improve this paper, 
both with the mathematics and clarity of exposition. 

The thesis was partly finished while attending the Mittag-Leffler institute. 
I am indebted to the Royal Swedish Academy of Sciences 
for financial support towards attending the institute during spring 2009.

In the midst of writing this thesis I became a father. 
A most loving acknowledgement to my wonderful wife Ida who has
looked after both me and our amazing little girl, Emma, during this
time.

\end{document}